\documentclass[12pt]{amsart}

\usepackage[utf8]{inputenc}
\usepackage[T1]{fontenc}
\usepackage{amsmath}
\usepackage{mathtools}
\usepackage{amssymb}
\usepackage{MnSymbol}
\usepackage{stmaryrd}
\usepackage{relsize}
\usepackage{graphicx}
\usepackage{verbatim}
\usepackage{amsthm}
\usepackage{bbm}
\usepackage{tikz}
\usepackage{float}
\usepackage{accents}

\makeatletter
\@namedef{subjclassname@2010}{%
  \textup{2010} Mathematics Subject Classification}
\makeatother

\newtheorem{theorem}{Theorem}[section]
\newtheorem{corollary}[theorem]{Corollary}
\newtheorem{lemma}[theorem]{Lemma}
\newtheorem{proposition}[theorem]{Proposition}
\theoremstyle{definition}
\newtheorem{definition}[theorem]{Definition}
\newtheorem{remark}[theorem]{Remark}
\newtheorem{example}[theorem]{Example}

\numberwithin{equation}{section}

\allowdisplaybreaks

\newcommand{\id}{\textnormal{id}}
\def\card#1{ | #1 | }
\newcommand{\noncr}{\mathcal{NC}}
\def\noncrk#1{\mathcal{NC}^{#1}}

\DeclareRobustCommand{\pairPartitionRecurrence}[2]
{
   	\vcenter
   	{	
   		\hbox
   		{
   			\tikz
   			{
				\tikzstyle{leg} = [circle, draw=black, fill=black!100, text=black!100, thin, inner sep=0pt, minimum size=2.0]

				\draw[-] (0,0) -- (0,0.5);
				\node () at (0,0) [leg] {.};
				\draw[-] (0.75,0) -- (0.75,0.5);
				\node () at (0.75,0) [leg] {.};
				\draw[-] (0,0.5) -- (0.75,0.5);
				\node () at (0.375,0) [label={[label distance = -1.5mm]above:\small{$#1$}}] {};
				\node () at (1.125,0) [label={[label distance = -1.5mm]above:\small{$#2$}}] {};
			}
		}
	} \! \! \! \!
}

\def\sgn{\textnormal{sgn}\,}
\def\unitInTheAlgebra{\mathbf{1}}
\def\ovcirc#1{\accentset{\circ}{\mathcal{#1}}}

\def\dirsum#1#2{\bigoplus \limits_{#1}^{#2}}
\def\vProduct{\mathop{ \mathlarger{\mathlarger{\mathlarger{\varovee}}} }}
\newcommand{\bound}{B}
\newcommand{\lsim}{\sim}
\newcommand{\rsim}{\sim}

\newcommand{\nrsim}{\nsim}
\newcommand{\diff}{d} 

\begin{document}

\baselineskip=17pt

\title[V-monotone independence]{V-monotone independence}

\author[A. Dacko]{Adrian Dacko}
\address{Faculty of Pure and Applied Mathematics\\ Wrocław University of Science and Technology\\
Wybrzeże Wyspiańskiego 27 \\
50-370 Wrocław, Poland}
\email{adrian.dacko@pwr.edu.pl}

\date{\today}


\begin{abstract}
We introduce and study a new notion of non-commutative independence, called V-monotone independence, which can be viewed as an extension of the monotone independence of Muraki. We investigate the combinatorics of mixed moments of V-monotone random variables and prove the central limit theorem. We obtain a combinatorial formula for the limit moments and we find the solution of the differential equation for the moment generating function in the implicit form.
\end{abstract}

\subjclass[2010]{Primary: 46L53. Secondary: 05A18.}

\keywords{Non-commutative probability, monotone independence, anti-monotone independence, V-monotone independence, V-monotone product of states, V-monotone central limit theorem}

\maketitle


\section{Introduction}
There are several notions of non-commutative independence, including classical independence, free independence~\cite{VoiculescuDykemaNica1992} and Boolean independence~\cite{{Bozejko1987},{SpeicherWoroudi1995}}. With each of these notions one can associate a product of states. According to the original version of the axiomatic approach~\cite{BenGhorbalSchurmann2002}, this product should satisfy four axioms. It was shown there that only the products associated with these three notions meet the four requirements. 

However, Muraki introduced a new concept of non-commutative independence, called monotone (see~\cite{{Muraki1997}, {Muraki2000}, {Muraki2001}}), which does not satisfy one of the axioms (commutativity). In this case, not arbitrary families of states, but only those indexed by totally ordered sets, can be considered. Despite these facts, the concept of monotone independence has attracted a considerable interest. Let us also remark that there is a related notion of the anti-monotone independence which can be obtained from the monotone independence by reversing the order of the index set. In a new version of the axiomatic approach, with the commutativity axiom dropped, Muraki showed in~\cite{Muraki2003} that there exist only five ``natural'' products of states, namely: classical, free, Boolean, monotone, and anti-monotone. Our product, associated with the V-monotone independence, satisfies two of four axioms: it fulfills the axioms of universality and normalization, but it is neither commutative, nor associative.

Other types of non-commutative independence have also been studied in the literature (see for instance~\cite{{Bozejko1991}, {BozejkoLeinertSpeicher1996}, {Hasebe2011}, {LenczewskiSalapata2006}, {Lenczewski2010}, {Wysoczanski2007}}). Notions of independence were also unified in certain ways (see~\cite{{Franz2003}, {Lenczewski1998}, {Lenczewski2010}}). Interpolations between ``axiomatic'' notions (and also between independences themselves) were also studied (see~\cite{{BozejkoWysoczanski2001}, {Hasebe2011},  {LenczewskiSalapata2006}, {LenczewskiSalapata2008}, {Wysoczanski2007}}). All types of non-commutative independence are related to many interesting topics, such as Fock spaces, operator algebras, convolutions of measures, combinatorics of partitions and trees.

Our model can be viewed as an extension of both monotone and anti-monotone models. In particular, the monotone and the anti-monotone Fock spaces are the subspaces of the V-monotone Fock space. Moreover, the class of V-monotonically labeled partitions contains the classes of monotonically labeled and anti-monotonically labeled partitions.

From a formal point of view, our definition resembles the language used in~\cite{LenczewskiSalapata2006}. We discuss the combinatorics, the associated product of states, and the central limit theorem. We express the moment generating function for the central limit law in the implicit form (for that purpose, we solve a certain Abel differential equation).  It is not clear whether one can find the limit law in the explicit form. Moreover, the existing enumeration results on ordered labeled trees do not seem to contain V-monotone labelings, which confirms our supposition that handling the limit law is highly non-trivial from the technical point of view. Therefore, this study seems to be worthy of continuation.

\section{Combinatorial definitions and notions}
For $n, m \in \mathbb{N}$, we denote $[n, n+m] = \{ n, n+1, \ldots, n+m \}$ and $[n] = [1,n]$. Let $X \subseteq \mathbb{N}_{+}$ be a set of cardinality $n \geq 0$ of the form $\{ x_1 < \ldots < x_n \}$ (for $n > 0$). By a \emph{partition} on $X$ we mean any family $\pi \subseteq 2^X$ such that for each $x \in X$ there exists exactly one $B \in \pi$ such that $x \in B$. 

Let $\mathcal{P}(X)$ denote the set of all partitions on $X$ (if $X = [n]$, we shall simply write $\mathcal{P}(n)$). The elements of $\pi \in \mathcal{P}(X)$ are called \emph{blocks} and the elements of any $B \in \pi$ are named \emph{legs}. If $\card{B} = 1$, then its leg and $B$ are called a \emph{singleton leg} and a \emph{singleton block}, respectively. If $B$ is of the form $\{x_{k}, x_{k+1},\ldots, x_{k+l}\}$, then it is named an \emph{interval}. If a block is not a singleton block, then its legs are divided into the leftmost, middle and the rightmost legs. The \emph{rightmost leg} of a block $B$ is the greatest element in $B$, the \emph{leftmost leg} of a block $B$ is the least element in $B$, and the remaining legs are referred to as \emph{middle legs}. 

Let $I$ be an arbitrary set of indices. By a \emph{labeling} of a partition $\pi$ we mean any function $\mathfrak{i} \colon \pi \mapsto I$. The pair $(\pi, \mathfrak{i})$ is called a \emph{labeled partition}. 

\begin{definition}
\label{definition:adaptabilityOfASequenceToAPartition}
We say that a sequence of indices $(i_1, \ldots, i_n)$ is \emph{adapted} to a labeled partition $(\pi, \mathfrak{i})$ if all legs of a given block $B$ have the same label $\mathfrak{i}(B)$.
\end{definition}

We illustrate Definition~\ref{definition:adaptabilityOfASequenceToAPartition} with the following example.

\begin{example}
The sequence $(3,2,3,4,3,2,3,5,5,8)$ is not adapted to any of labeled partition, presented in Fig.~\ref{figure:labeledPartitions}, whereas $(3,2,3,4,3,2,3,5,5,3)$ is adapted to the one on the left hand side and not adapted to the one on the right hand side.

\begin{figure}[H]
\centering
\begin{tikzpicture}
	\tikzstyle{leg} = [circle, draw=black, fill=black!100, text=black!100, thin, inner sep=0pt, minimum size=4.0]

    \pgfmathsetmacro {\dx}{0.25}
    \pgfmathsetmacro {\dy}{0.40}

    \pgfmathsetmacro {\x}{0}
    \pgfmathsetmacro {\y}{0}

	\draw[-] (\x+0*\dx,\y+0*\dy) -- (\x+0*\dx,\y+4*\dy);
	\node () at (\x+0*\dx,\y+0*\dy) [leg] {.};
	\draw[-] (\x+6*\dx,\y+0*\dy) -- (\x+6*\dx,\y+4*\dy);
	\node () at (\x+6*\dx,\y+0*\dy) [leg] {.};
	\draw[-] (\x+9*\dx,\y+0*\dy) -- (\x+9*\dx,\y+4*\dy);
	\node () at (\x+9*\dx,\y+0*\dy) [leg] {.};
	\draw[-] (\x+0*\dx,\y+4*\dy) -- (\x+9*\dx,\y+4*\dy);
	\node () at (\x+4.5*\dx,\y+4*\dy) [label={[label distance = -2mm]above:\smaller{3}}] {};
	
	\draw[-] (\x+1*\dx,\y+0*\dy) -- (\x+1*\dx,\y+3*\dy);
	\node () at (\x+1*\dx,\y+0*\dy) [leg] {.};
	\draw[-] (\x+5*\dx,\y+0*\dy) -- (\x+5*\dx,\y+3*\dy);
	\node () at (\x+5*\dx,\y+0*\dy) [leg] {.};
	\draw[-] (\x+1*\dx,\y+3*\dy) -- (\x+5*\dx,\y+3*\dy);
	\node () at (\x+3*\dx,\y+3*\dy) [label={[label distance = -2mm]above:\smaller{2}}] {};
	
	\draw[-] (\x+2*\dx,\y+0*\dy) -- (\x+2*\dx,\y+2*\dy);
	\node () at (\x+2*\dx,\y+0*\dy) [leg] {.};
	\draw[-] (\x+4*\dx,\y+0*\dy) -- (\x+4*\dx,\y+2*\dy);
	\node () at (\x+4*\dx,\y+0*\dy) [leg] {.};
	\draw[-] (\x+2*\dx,\y+2*\dy) -- (\x+4*\dx,\y+2*\dy);
	\node () at (\x+3*\dx,\y+2*\dy) [label={[label distance = -2mm]above:\smaller{3}}] {};
	
	\draw[-] (\x+3*\dx,\y+0*\dy) -- (\x+3*\dx,\y+1*\dy);
	\node () at (\x+3*\dx,\y+0*\dy) [leg] {.};
	\node () at (\x+3*\dx,\y+1*\dy) [label={[label distance = -2mm]above:\smaller{4}}] {};
	
	\draw[-] (\x+7*\dx,\y+0*\dy) -- (\x+7*\dx,\y+1*\dy);
	\node () at (\x+7*\dx,\y+0*\dy) [leg] {.};
	\draw[-] (\x+8*\dx,\y+0*\dy) -- (\x+8*\dx,\y+1*\dy);
	\node () at (\x+8*\dx,\y+0*\dy) [leg] {.};
	\draw[-] (\x+7*\dx,\y+1*\dy) -- (\x+8*\dx,\y+1*\dy);
	\node () at (\x+7.5*\dx,\y+1*\dy) [label={[label distance = -2mm]above:\smaller{5}}] {};
	
    \pgfmathsetmacro {\x}{5}
    \pgfmathsetmacro {\y}{0}

	\draw[-] (\x+0*\dx,\y+0*\dy) -- (\x+0*\dx,\y+4*\dy);
	\node () at (\x+0*\dx,\y+0*\dy) [leg] {.};
	\draw[-] (\x+6*\dx,\y+0*\dy) -- (\x+6*\dx,\y+4*\dy);
	\node () at (\x+6*\dx,\y+0*\dy) [leg] {.};
	\draw[-] (\x+9*\dx,\y+0*\dy) -- (\x+9*\dx,\y+4*\dy);
	\node () at (\x+9*\dx,\y+0*\dy) [leg] {.};
	\draw[-] (\x+0*\dx,\y+4*\dy) -- (\x+9*\dx,\y+4*\dy);
	\node () at (\x+4.5*\dx,\y+4*\dy) [label={[label distance = -2mm]above:\smaller{1}}] {};
	
	\draw[-] (\x+1*\dx,\y+0*\dy) -- (\x+1*\dx,\y+3*\dy);
	\node () at (\x+1*\dx,\y+0*\dy) [leg] {.};
	\draw[-] (\x+5*\dx,\y+0*\dy) -- (\x+5*\dx,\y+3*\dy);
	\node () at (\x+5*\dx,\y+0*\dy) [leg] {.};
	\draw[-] (\x+1*\dx,\y+3*\dy) -- (\x+5*\dx,\y+3*\dy);
	\node () at (\x+3*\dx,\y+3*\dy) [label={[label distance = -2mm]above:\smaller{3}}] {};
	
	\draw[-] (\x+2*\dx,\y+0*\dy) -- (\x+2*\dx,\y+2*\dy);
	\node () at (\x+2*\dx,\y+0*\dy) [leg] {.};
	\draw[-] (\x+4*\dx,\y+0*\dy) -- (\x+4*\dx,\y+2*\dy);
	\node () at (\x+4*\dx,\y+0*\dy) [leg] {.};
	\draw[-] (\x+2*\dx,\y+2*\dy) -- (\x+4*\dx,\y+2*\dy);
	\node () at (\x+3*\dx,\y+2*\dy) [label={[label distance = -2mm]above:\smaller{2}}] {};
	
	\draw[-] (\x+3*\dx,\y+0*\dy) -- (\x+3*\dx,\y+1*\dy);
	\node () at (\x+3*\dx,\y+0*\dy) [leg] {.};
	\node () at (\x+3*\dx,\y+1*\dy) [label={[label distance = -2mm]above:\smaller{4}}] {};
	
	\draw[-] (\x+7*\dx,\y+0*\dy) -- (\x+7*\dx,\y+1*\dy);
	\node () at (\x+7*\dx,\y+0*\dy) [leg] {.};
	\draw[-] (\x+8*\dx,\y+0*\dy) -- (\x+8*\dx,\y+1*\dy);
	\node () at (\x+8*\dx,\y+0*\dy) [leg] {.};
	\draw[-] (\x+7*\dx,\y+1*\dy) -- (\x+8*\dx,\y+1*\dy);
	\node () at (\x+7.5*\dx,\y+1*\dy) [label={[label distance = -2mm]above:\smaller{9}}] {};
	
	\node () at (0, -0.25) {};
\end{tikzpicture}
\caption{Labeled partitions}
\label{figure:labeledPartitions}
\end{figure}
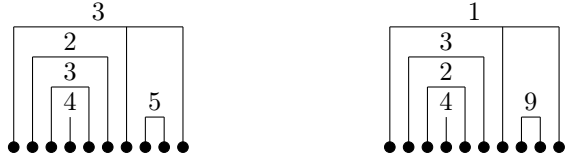
\end{example}

In the definition below we assume that $I$ is the set of all natural numbers.
\begin{definition}
\label{definition:orderedPartitions}
Let $\mathcal{OP}(n)$ be the set of all labeled partitions $(\pi, \mathfrak{i})$ such that $\mathfrak{i}$ is a bijection between $\pi$ and $[\card{\pi}]$ (here $\card{\pi}$ denotes the number of blocks in $\pi$), i.e. it is the set of all \emph{ordered partitions} on $[n]$.
\end{definition}

Observe that for each partition $\pi \in \mathcal{P}(n)$ there exists exactly $|\pi|!$ labelings $\mathfrak{i}$ such that $(\pi, \mathfrak{i}) \in \mathcal{OP}(n)$. Here, the order is not an additional structure, but it is implemented by a labeling.

We say that blocks $B$ and $B'$ have a \emph{crossing} if there exist $l_1, l_2 \in B$ and $l'_1, l'_2 \in B'$ such that $l_1 < l'_1 < l_2 < l'_2$ or $l'_1 < l_1 < l'_2 < l_2$. A partition without blocks with crossings is called a \emph{non-crossing partition}. The set of all non-crossing partitions on $X = \{ x_1 < \ldots < x_n \}$ will be denoted by $\noncr(X)$ ($\noncr(n)$ if $X=[n]$). 

Let $\noncr = \bigcup \limits_{n=0}^\infty \noncr(n)$ and let $\noncr^2 = \bigcup \limits_{n=0}^\infty \noncrk{2}(2n)$, where $\noncrk{2}(2n)$ is the set of all non-crossing pair partitions on $[2n]$, i.e. each block of $\pi \in \noncrk{2}(2n)$ has exactly two legs. In the special case $n=0$, by $\noncr(0)$ and $\noncrk{2}(0)$ we mean the set $\{ \emptyset \}$.

For $\pi \in \noncr$, we say that a block $B \in \pi$ is \emph{inner} with respect to a block $B' = \{ p_1 < \ldots < p_q \} \in \pi$ and we write $B' < B$ if there exists $i \in [q-1]$ such that for any leg $k \in B$ we have $p_i < k < p_{i+1}$. In such case we say that $B'$ is \emph{outer} with respect to $B$. We say that $B'$ is the \emph{nearest outer block} with respect to $B$ and we write $B' \prec B$ if $B' < B$ and there is no block $B''$ such that $B' < B'' < B$ (see~\cite{Lenczewski2010}). We say that blocks $B$ and $B'$ are \emph{neighboring} if one of them is the nearest outer block with respect to the other. If a block $B$ has no outer blocks, we call it an \emph{outer block} in $\pi$. For more information about non-crossing partitions, see~\cite{NicaSpeicher2006}.

\section{V-monotone independence}
\label{section:VmonotoneIndependence}
Throughout Sections~\ref{section:VmonotoneIndependence} and~\ref{section:hilbertSpaceRealization}, by $(I, \leq)$ we will understand an arbitrary totally ordered set. In this section, $(\mathcal{A}, \varphi)$ stands for a non-commutative probability space (i.e. $\mathcal{A}$ is a unital algebra with the unit $\unitInTheAlgebra$ and $\varphi$ is a state, that is a linear functional $\varphi \colon \mathcal{A} \mapsto \mathbb{C}$ such that $\varphi(\unitInTheAlgebra) = 1$) and $(\mathcal{A}_i)_{i \in I}$ for a family of subalgebras of $\mathcal{A}$. We assume that there exists a family $(\unitInTheAlgebra_i)_{i \in I}$ such that for any $i \in I$ the element $\unitInTheAlgebra_i$ is the inner unit of $\mathcal{A}_i$, i.e. for any $a \in \mathcal{A}_i$ we have $\unitInTheAlgebra_i a = a \unitInTheAlgebra_i = a$.

\begin{definition}
\label{definition:I_n-indices}
For $1 \leq m \leq n$ let $I_{n,m} = \{ (i_1, \ldots, i_n) \in I^n: i_1 > \ldots > i_m < \ldots < i_n \}$ (we put $I_{1,1} = I$) and $I_n \coloneqq \bigcup \limits_{m=1}^n I_{n,m}$. For $i \in I$ and $(i_1, \ldots, i_n) \in I_n$ we write $i \lsim (i_1, \ldots, i_n)$ if $(i, i_1, \ldots, i_n) \in I_{n+1}$ and $(i_1, \ldots, i_n) \rsim i$ if $(i_1, \ldots, i_n, i) \in I_{n+1}$.
\end{definition}

\begin{definition}
\label{definition:definitionOfVmonotoneIndependence}
We say that the family $(\mathcal{A}_i)_{i \in I}$ is \emph{V-monotonically independent} with respect to $\varphi$ if for any $i \in I$ we have $\varphi(\unitInTheAlgebra_i) = 1$, and for any sequence $(i_1, \ldots, i_n) \in I^n$ such that $i_k \neq i_{k+1}$ for $k \in [n-1]$ and for any $a_1 \in \mathcal{A}_{i_1}, \ldots, a_n \in \mathcal{A}_{i_n}$  the following conditions are satisfied:
\begin{itemize}
	\item[ (i)] $\varphi(a_1 \ldots a_n) = 0$, if $\varphi(a_1) = \ldots = \varphi(a_n) = 0$ (the freeness condition),
	\item[(ii)] for any $j \in [n]$ we have
	$$
	\varphi(a_1 \ldots a_{j-1} \unitInTheAlgebra_{i_j} a_{j+1} \ldots a_n) = 
	\begin{cases}
	\varphi(a_1 \ldots a_{j-1} a_{j+1} \ldots a_n) & \text{if $j \leq r$} \\
	0 & \text{otherwise,}
	\end{cases}
	$$
	whenever $\varphi(a_1) = \ldots = \varphi(a_{j-1}) = 0$ and $r \in [n]$ is such that $I_r \ni (i_1, \ldots, i_r) \nrsim i_{r+1}$ (without the condition with `$\nrsim$' for $r=n$).
\end{itemize}
We say that a family $(a_i)_{i \in I}$ of non-commutative random variables from $\mathcal{A}$ is V-monotone if the family of algebras generated by $a_i$ and $\unitInTheAlgebra_i$, for each $i \in I$, is V-monotone.
\end{definition}

This definition can be viewed as follows: on the Fock space level, $\unitInTheAlgebra_{i}$ is the projection onto the subspace onto which the algebra $\mathcal{A}_i$ acts non-trivially. The letter V in the word
``V-monotone'' comes from the shape of the graph of a non-monotonic sequence from $I_{n}$ which resembles the letter V.

\begin{remark}
\label{remark:axiomsOfNoncommutativeIndependence}
As we mentioned in Introduction, the product of states (Definition~\ref{definition:definitionOfVmonotoneProductOfStates}) associated with the above definition of independence fulfills two of four axioms stated in~\cite{BenGhorbalSchurmann2002} (and slightly modified in~\cite{Muraki2003}, see Definition~2.1 therein). Namely, the V-monotone product of states is neither commutative (see Example~\ref{example:firstMixedMomentsForTwoAlgebras}) nor associative (see Proposition~\ref{proposition:notAssociative}). However, the axiom of universality is fulfilled, which follows from Corollary~\ref{corollary:universalPolynomialFromRecursion}, and also the axioms of normalization (axioms (U4) in~\cite{Muraki2003}) hold: the first one (extension) is satisfied by the definition and the second one (factorization) follows from Example~\ref{example:firstMixedMomentsForTwoAlgebras}. Moreover, the V-monotone product of more than two states meets the proper generalizations of the latter two requirements.
\end{remark}

The conditions stated in Definition~\ref{definition:definitionOfVmonotoneIndependence} give a recursive formula for any mixed moment.
\begin{proposition}
\label{proposition:recursiveFormulaForMixedMoments}
We assume that $(\mathcal{A}_i)_{i \in I}$ is V-monotone independent with respect to $\varphi$, $(i_1, \ldots, i_n) \in I^n$ is such that neighboring terms are different, and $r \in [n]$ is such that $I_r \ni (i_1, \ldots, i_r) \nrsim i_{r+1}$ (if $(i_1, \ldots, i_n) \in I_n$, then we set $r=n$). Let $a_1 \in \mathcal{A}_{i_1}, \ldots, a_n \in \mathcal{A}_{i_n}$. For $k \in [n]$ and $\varepsilon \colon [n] \mapsto \{0,1\}$ define
$$
a_k^{\varepsilon(k)} = 
\begin{cases}
a_k - \varphi(a_k) \unitInTheAlgebra_{i_k} & \text{if $\varepsilon(k) = 0$} \\
\varphi(a_k) \unitInTheAlgebra_{i_k} & \text{if $\varepsilon(k) = 1$.}
\end{cases}
$$
The following recursive formula holds:
\begin{equation}
\label{eq:recursiveFormulaForMixedMoments }
\varphi(a_1 \ldots a_n) = \sum \limits_{k=1}^r \varphi(a_k) \varphi(a_1^0 \ldots a_{k-1}^0 a_{k+1} \ldots a_n) \text{,}
\end{equation}
where $\varphi(a_1^0 \ldots a_{k-1}^0 a_{k+1} \ldots a_n) = \varphi(a_2 a_3 \ldots a_n)$ for $k=1$.
\end{proposition}

\begin{proof}
\begin{equation*}
\begin{split}
\varphi(a_1 \ldots a_n) 
& = \varphi( (a_1^0 + a_1^1) \ldots (a_n^0 + a_n^1) ) = \sum \limits_{ \varepsilon \in \{0,1\}^n } \varphi(a_1^{\varepsilon(1)} \ldots a_n^{\varepsilon(n)}) \\
& = \varphi(a_1^0 \ldots a_n^0) + \sum \limits_{j=1}^n \varphi(a_j) \sum \limits_{ \varepsilon \in \{0,1\}^{n-j}} \varphi(a_1^0 \ldots a_{j-1}^0 \unitInTheAlgebra_{i_j} a_{j+1}^{\varepsilon(j+1)} \ldots a_n^{\varepsilon(n)}) \text{.}
\end{split}
\end{equation*}
Note that V-monotone independence yields
\begin{align*}
\varphi(a_1 \ldots a_n)
& = \sum \limits_{j=1}^r \varphi(a_j) \sum \limits_{ \varepsilon \in \{0,1\}^{n-j} } \varphi(a_1^0 \ldots a_{j-1}^0 a_{j+1}^{\varepsilon(j+1)} \ldots a_n^{\varepsilon(n)}) \\
& = \sum \limits_{j=1}^r \varphi(a_j) \varphi(a_1^0 \ldots a_{j-1}^0 a_{j+1} \ldots a_n) \text{,}
\end{align*}
which establishes the desired formula.
\end{proof}

\begin{corollary}
\label{corollary:universalPolynomialFromRecursion}
For any $i_1 \neq \ldots \neq i_n$ there exists a polynomial $w$ of variables $\{ x_B: B \in \mathcal{I}(i_1, \ldots, i_n) \}$, where 
$$
\mathcal{I}(i_1, \ldots, i_n) \coloneqq \{ B \subseteq [n] : i_k = i_l \text{ for any } k, l \in B \} \text{,}
$$
such that for any $a_1 \in \mathcal{A}_{i_1}, \ldots, a_n \in \mathcal{A}_{i_n}$ we have
$$
\varphi(a_1 \ldots a_n) = w(\varphi(a_B): B \in \mathcal{I}(i_1, \ldots, i_n)) \text{,}
$$
where $a_B = a_{p_1} \ldots a_{p_q}$ for $B = \{ p_1 < \ldots < p_q \}$.
\end{corollary}

\begin{example}
\label{example:firstMixedMomentsForTwoAlgebras}
Easy computations based on Proposition~\ref{proposition:recursiveFormulaForMixedMoments} give that all mixed moments up to order five are the same as the corresponding mixed moments in the free case (cf. Lecture~5 in~\cite{NicaSpeicher2006} or~\cite{VoiculescuDykemaNica1992}). Now, let $a_1, a_2, a_3$, and $b_1, b_2, b_3$ be random variables from $\mathcal{A}_1$ and $\mathcal{A}_2$, respectively. Then
\begin{align*}
\varphi(a_1 b_1 a_2 b_2 a_3) = &
\ \varphi(a_1 a_2 a_3) \varphi(b_1) \varphi(b_2) + \varphi(a_1) \varphi(a_2) \varphi(a_3) \varphi(b_1 b_2) \\
& - \varphi(a_1) \varphi(a_2) \varphi(a_3) \varphi(b_1) \varphi(b_2)
\end{align*}
and
\begin{align*}
\varphi(b_1 a_1 b_2 a_2 b_3) = & \
\varphi(a_1 a_2) \varphi(b_1 b_3) \varphi(b_2) + \varphi(a_1) \varphi(a_2) \varphi(b_1 b_2 b_3) \\
& - \varphi(a_1) \varphi(a_2) \varphi(b_1 b_3) \varphi(b_2) \text{.}
\end{align*}
These mixed moments differ from the corresponding mixed moments of free, Boolean, monotone, and anti-monotone random variables. 

We only prove the latter equality. Using Proposition~\ref{proposition:recursiveFormulaForMixedMoments} and the formulas for mixed moments of order less than five, we have
\begin{align*}
\varphi(b_1 a_1 b_2 a_2 b_3) = & \ \varphi(b_1) \varphi(a_1 b_2 a_2 b_3) + \varphi(a_1) \varphi(b_1^{0} b_2 a_2 b_3) + \varphi(b_2) \varphi(b_1^{0} a_1^{0} a_2 b_3) \\
= & \ \varphi(b_1) [ \varphi(a_1 a_2) \varphi(b_2) \varphi(b_3) + \varphi(a_1) \varphi(a_2) \varphi(b_2 b_3) - \varphi(a_1) \varphi(a_2) \varphi(b_2) \varphi(b_3) ] \\
& + \varphi(a_1) \varphi(a_2) \varphi(b_1^{0} b_2 b_3) + \varphi(b_2) \varphi(a_1^{0} a_2) \varphi(b_1^{0} b_3) \\
= & \ \varphi(a_1 a_2) \varphi(b_1 b_3) \varphi(b_2) + \varphi(a_1) \varphi(a_2) \varphi(b_1 b_2 b_3) - \varphi(a_1) \varphi(a_2) \varphi(b_1 b_3) \varphi(b_2) \text{.}
\end{align*}

Referring to Corollary~\ref{corollary:universalPolynomialFromRecursion}, the polynomial $w$ associated with $(1,2,1,2,1)$ has the form
$$
x_{135} x_{2} x_{4} + x_{1} x_{3} x_{5} x_{24} - x_{1} x_{3} x_{5} x_{2} x_{4}
$$
and the polynomial $\tilde{w}$, associated with $(2,1,2,1,2)$ has the form
$$
x_{24} x_{15} x_{3} + x_{2} x_{4} x_{135} - x_{2} x_{4} x_{15} x_{3} \text{}
$$
(here we write $x_{p_1, \ldots, p_q}$ instead of $x_{ \{ p_1 < \ldots < p_q \} }$). Note that these polynomials have integer coefficients and also observe that not all variables are used. For example, the variable $x_{35}$ does not appear in polynomial $w$.
\end{example}

Now, we find a nice characterization of V-monotone independence, using a combinatorial formula for mixed moments of random variables.

\begin{definition}
\label{definition:kappaStarFunctionals}
We define the family of multilinear functionals $(\kappa_n^*)_{n=1}^{\infty}$ recursively by
$$
\begin{cases}
\kappa_{1}^*(a) = \varphi(a) \\
\kappa_{n+1}^*(a_1, \ldots, a_{n+1}) = \kappa_{n}^*(a_1 a_2, a_3, \ldots, a_{n+1}) - \varphi(a_1) \kappa_{n}^*(a_2, a_3, \ldots, a_{n+1}) \text{.}
\end{cases}
$$
\end{definition}

\begin{proposition}
\label{proposition:kappaStarFunctionalsProperties}
For all $n \in \mathbb{N}_{+}$ and $k \in [n]$ 
\begin{multline*}
\kappa_{n+1}^*(a_1, \ldots, a_k, a_{k+1}, \ldots, a_{n+1}) \\
= \kappa_{n}^*(a_1, \ldots, a_k a_{k+1}, \ldots, a_{n+1}) - \kappa_{k}^*(a_1, \ldots, a_k) \kappa_{n+1-k}^*(a_{k+1}, \ldots, a_{n+1}) \text{.}
\end{multline*}
Moreover, in the case when $\mathcal{A} = \bound(\mathcal{H})$ for some Hilbert space $\mathcal{H}$ and vector state $\varphi$ associated with $\xi \in \mathcal{H}$, we have
$$
\kappa_{n}^*(a_1, \ldots, a_n) = \varphi(a_1 P^{\perp} a_2 P^{\perp} \ldots P^{\perp} a_n) \text{,}
$$
where $P^{\perp}$ is the orthogonal projection onto $(\mathbb{C} \, \xi)^{\perp}$. 
\end{proposition}

\begin{proof}
The proof of the first claim is by induction on $n$. Using the definition and the induction hypothesis, we get
\begin{align*}
& \kappa^*_{n+1}(a_1, \ldots, a_k, a_{k+1}, \ldots, a_{n+1}) = \kappa^*_{n-1}(b_1, \ldots, b_k, \ldots, a_{n+1}) \\
& - \kappa^*_{k-1}(b_1, \ldots, a_k) \kappa^*_{n+1-k}(a_{k+1}, \ldots, a_{n+1}) - \varphi(a_1)\kappa^*_{n-1}(a_2, \ldots, b_k, \ldots, a_{n+1}) \\
& + \varphi(a_1) \kappa^*_{k-1}(a_2, \ldots, a_k) \kappa^*_{n+1-k}(a_{k+1}, \ldots, a_{n+1}) \\
= & \ \kappa^*_n(a_1, \ldots, b_k, \ldots, a_{n+1}) - \kappa^*_{k}(a_1, \ldots, a_k) \kappa^*_{n+1-k}(a_{k+1}, \ldots, a_{n+1}) \text{,}
\end{align*}
where, for convenience, we write $b_l = a_l a_{l+1}$, $l \in [n]$. 

Now, observe that for any $a, b \in \bound(\mathcal{H})$ we have $\varphi(a P b) = \varphi(a) \varphi(b)$, where $P = 1-P^{\perp}$. Indeed,
$$
a P b \, \xi = a P \, (\varphi(b) \xi + (b \, \xi)^{\perp}) = \varphi(b) a \, \xi = \varphi(b) (\varphi(a) \xi + (a \, \xi)^{\perp}) \text{.}
$$
Hence we have
\begin{align*}
\varphi(a_1 P^{\perp} a_2 P^{\perp} \ldots P^{\perp} a_{n+1}) 
& = \varphi(a_1 (1-P) a_2 P^{\perp} \ldots P^{\perp} a_{n+1}) \\
& = \varphi(a_1 a_2 P^{\perp} \ldots P^{\perp} a_{n+1}) - \varphi(a_1) \varphi(a_2 P^{\perp} \ldots P^{\perp} a_{n+1}) \text{.}
\end{align*}
\end{proof}

\begin{remark}
For non-negative integers $m_1, \ldots, m_n$ and a random variable $a \in \mathcal{A}$ with a distribution $\mu$, we have
$$
\kappa^*_n(a^{m_1}, \ldots, a^{m_n}) = (-1)^{|\pi|-1} k^*_{\mu}(\pi) \text{,}
$$
where $\pi$ is the interval partition on $[m_1 + \ldots + m_n]$, defined by the sequence $(m_1, \ldots, m_n)$. The functions $k^*_{\mu}$, called \emph{inverse Boolean cumulant functions}, appeared in~\cite{Lenczewski2007} in the context of the so-called orthogonal convolution.
\end{remark}

\begin{definition}
\label{definition:VmonotoneLabelling}
We say that $(\pi, \mathfrak{i})$ is \emph{V-monotonically labeled} if for all $\{ B_1 \succ \ldots \succ B_r \} \subseteq \pi$ we have $(\mathfrak{i}(B_1), \ldots, \mathfrak{i}(B_r)) \in I_r$. By $\mathcal{V}(i_1, \ldots, i_n)$ we denote the set of all V-monotonically labeled non-crossing partitions $(\pi, \mathfrak{i})$ to which $(i_1, \ldots, i_n)$ is adapted. 
\end{definition}

We illustrate the definition above with two following examples.

\begin{example}
The partition on the left hand side in Fig.~\ref{figure:labeledPartitions} is V-monotonically labeled (because $(3,2,3,4) \in I_4$ and $(3,5) \in I_2$), whereas the one on the right hand side is not (since $(1,3,2,4) \notin I_4$).
\end{example}

\begin{example}
Consider the sequence $(2,7,5,7,5,2)$. One can construct six non-crossing labeled partitions adapted to it and they are shown in Fig.~\ref{figure:adaptedVMonotonePartitions}. Only the one in the lower right corner is not V-monotonically labeled, thus $\mathcal{V}(2,7,5,7,5,2)$ consists of five remaining labeled partitions.

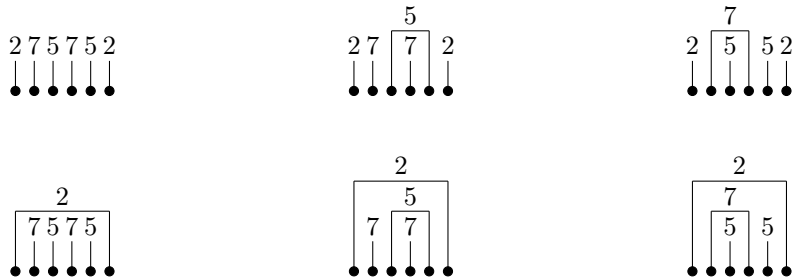
\begin{figure}[H]
\centering
\begin{tikzpicture}
    \tikzstyle{leg} = [circle, draw=black, fill=black!100, text=black!100, thin, inner sep=0pt, minimum size=2.0]

    \pgfmathsetmacro {\dx}{0.25}
    \pgfmathsetmacro {\dy}{0.40}

    \pgfmathsetmacro {\x}{0*\dx}
    \pgfmathsetmacro {\y}{0}

    \draw[-] (\x+0*\dx,\y+0*\dy) -- (\x+0*\dx,\y+1*\dy);
    \node () at (\x+0*\dx,\y+0*\dy) [leg] {.};
    \node () at (\x+0*\dx,\y+1*\dy) [label={[label distance = -2mm]above:\smaller{2}}] {};
    
    \draw[-] (\x+1*\dx,\y+0*\dy) -- (\x+1*\dx,\y+1*\dy);
    \node () at (\x+1*\dx,\y+0*\dy) [leg] {.};
    \node () at (\x+1*\dx,\y+1*\dy) [label={[label distance = -2mm]above:\smaller{7}}] {};
    
    \draw[-] (\x+2*\dx,\y+0*\dy) -- (\x+2*\dx,\y+1*\dy);
    \node () at (\x+2*\dx,\y+0*\dy) [leg] {.};
    \node () at (\x+2*\dx,\y+1*\dy) [label={[label distance = -2mm]above:\smaller{5}}] {};
    
    \draw[-] (\x+3*\dx,\y+0*\dy) -- (\x+3*\dx,\y+1*\dy);
    \node () at (\x+3*\dx,\y+0*\dy) [leg] {.};
    \node () at (\x+3*\dx,\y+1*\dy) [label={[label distance = -2mm]above:\smaller{7}}] {};
    
    \draw[-] (\x+4*\dx,\y+0*\dy) -- (\x+4*\dx,\y+1*\dy);
    \node () at (\x+4*\dx,\y+0*\dy) [leg] {.};
    \node () at (\x+4*\dx,\y+1*\dy) [label={[label distance = -2mm]above:\smaller{5}}] {};
    
    \draw[-] (\x+5*\dx,\y+0*\dy) -- (\x+5*\dx,\y+1*\dy);
    \node () at (\x+5*\dx,\y+0*\dy) [leg] {.};
    \node () at (\x+5*\dx,\y+1*\dy) [label={[label distance = -2mm]above:\smaller{2}}] {};

    \pgfmathsetmacro {\x}{18*\dx}
    \pgfmathsetmacro {\y}{0}

    \draw[-] (\x+0*\dx,\y+0*\dy) -- (\x+0*\dx,\y+1*\dy);
    \node () at (\x+0*\dx,\y+0*\dy) [leg] {.};
    \node () at (\x+0*\dx,\y+1*\dy) [label={[label distance = -2mm]above:\smaller{2}}] {};
    
    \draw[-] (\x+1*\dx,\y+0*\dy) -- (\x+1*\dx,\y+1*\dy);
    \node () at (\x+1*\dx,\y+0*\dy) [leg] {.};
    \node () at (\x+1*\dx,\y+1*\dy) [label={[label distance = -2mm]above:\smaller{7}}] {};
    
    \draw[-] (\x+2*\dx,\y+0*\dy) -- (\x+2*\dx,\y+2*\dy);
    \node () at (\x+2*\dx,\y+0*\dy) [leg] {.};
    \draw[-] (\x+4*\dx,\y+0*\dy) -- (\x+4*\dx,\y+2*\dy);
    \node () at (\x+4*\dx,\y+0*\dy) [leg] {.};
    \draw[-] (\x+2*\dx,\y+2*\dy) -- (\x+4*\dx,\y+2*\dy);
    \node () at (\x+3*\dx,\y+2*\dy) [label={[label distance = -2mm]above:\smaller{5}}] {};
    
    \draw[-] (\x+3*\dx,\y+0*\dy) -- (\x+3*\dx,\y+1*\dy);
    \node () at (\x+3*\dx,\y+0*\dy) [leg] {.};
    \node () at (\x+3*\dx,\y+1*\dy) [label={[label distance = -2mm]above:\smaller{7}}] {};
    
    \draw[-] (\x+5*\dx,\y+0*\dy) -- (\x+5*\dx,\y+1*\dy);
    \node () at (\x+5*\dx,\y+0*\dy) [leg] {.};
    \node () at (\x+5*\dx,\y+1*\dy) [label={[label distance = -2mm]above:\smaller{2}}] {};

    \pgfmathsetmacro {\x}{36*\dx}
    \pgfmathsetmacro {\y}{0}

    \draw[-] (\x+0*\dx,\y+0*\dy) -- (\x+0*\dx,\y+1*\dy);
    \node () at (\x+0*\dx,\y+0*\dy) [leg] {.};
    \node () at (\x+0*\dx,\y+1*\dy) [label={[label distance = -2mm]above:\smaller{2}}] {};
    
    \draw[-] (\x+1*\dx,\y+0*\dy) -- (\x+1*\dx,\y+2*\dy);
    \node () at (\x+1*\dx,\y+0*\dy) [leg] {.};
    \draw[-] (\x+3*\dx,\y+0*\dy) -- (\x+3*\dx,\y+2*\dy);
    \node () at (\x+3*\dx,\y+0*\dy) [leg] {.};
    \draw[-] (\x+1*\dx,\y+2*\dy) -- (\x+3*\dx,\y+2*\dy);
    \node () at (\x+2*\dx,\y+2*\dy) [label={[label distance = -2mm]above:\smaller{7}}] {};
    
    \draw[-] (\x+2*\dx,\y+0*\dy) -- (\x+2*\dx,\y+1*\dy);
    \node () at (\x+2*\dx,\y+0*\dy) [leg] {.};
    \node () at (\x+2*\dx,\y+1*\dy) [label={[label distance = -2mm]above:\smaller{5}}] {};
    
    \draw[-] (\x+4*\dx,\y+0*\dy) -- (\x+4*\dx,\y+1*\dy);
    \node () at (\x+4*\dx,\y+0*\dy) [leg] {.};
    \node () at (\x+4*\dx,\y+1*\dy) [label={[label distance = -2mm]above:\smaller{5}}] {};
    
    \draw[-] (\x+5*\dx,\y+0*\dy) -- (\x+5*\dx,\y+1*\dy);
    \node () at (\x+5*\dx,\y+0*\dy) [leg] {.};
    \node () at (\x+5*\dx,\y+1*\dy) [label={[label distance = -2mm]above:\smaller{2}}] {};
    
    \pgfmathsetmacro {\x}{0}
    \pgfmathsetmacro {\y}{-6*\dy}

    \draw[-] (\x+0*\dx,\y+0*\dy) -- (\x+0*\dx,\y+2*\dy);
    \node () at (\x+0*\dx,\y+0*\dy) [leg] {.};
    \draw[-] (\x+5*\dx,\y+0*\dy) -- (\x+5*\dx,\y+2*\dy);
    \node () at (\x+5*\dx,\y+0*\dy) [leg] {.};
    \draw[-] (\x+0*\dx,\y+2*\dy) -- (\x+5*\dx,\y+2*\dy);
    \node () at (\x+2.5*\dx,\y+2*\dy) [label={[label distance = -2mm]above:\smaller{2}}] {};
    
    \draw[-] (\x+1*\dx,\y+0*\dy) -- (\x+1*\dx,\y+1*\dy);
    \node () at (\x+1*\dx,\y+0*\dy) [leg] {.};
    \node () at (\x+1*\dx,\y+1*\dy) [label={[label distance = -2mm]above:\smaller{7}}] {};
    
    \draw[-] (\x+2*\dx,\y+0*\dy) -- (\x+2*\dx,\y+1*\dy);
    \node () at (\x+2*\dx,\y+0*\dy) [leg] {.};
    \node () at (\x+2*\dx,\y+1*\dy) [label={[label distance = -2mm]above:\smaller{5}}] {};
    
    \draw[-] (\x+3*\dx,\y+0*\dy) -- (\x+3*\dx,\y+1*\dy);
    \node () at (\x+3*\dx,\y+0*\dy) [leg] {.};
    \node () at (\x+3*\dx,\y+1*\dy) [label={[label distance = -2mm]above:\smaller{7}}] {};
    
    \draw[-] (\x+4*\dx,\y+0*\dy) -- (\x+4*\dx,\y+1*\dy);
    \node () at (\x+4*\dx,\y+0*\dy) [leg] {.};
    \node () at (\x+4*\dx,\y+1*\dy) [label={[label distance = -2mm]above:\smaller{5}}] {};
    
    \pgfmathsetmacro {\x}{18*\dx}
    \pgfmathsetmacro {\y}{-6*\dy}

    \draw[-] (\x+0*\dx,\y+0*\dy) -- (\x+0*\dx,\y+3*\dy);
    \node () at (\x+0*\dx,\y+0*\dy) [leg] {.};
    \draw[-] (\x+5*\dx,\y+0*\dy) -- (\x+5*\dx,\y+3*\dy);
    \node () at (\x+5*\dx,\y+0*\dy) [leg] {.};
    \draw[-] (\x+0*\dx,\y+3*\dy) -- (\x+5*\dx,\y+3*\dy);
    \node () at (\x+2.5*\dx,\y+3*\dy) [label={[label distance = -2mm]above:\smaller{2}}] {};
    
    \draw[-] (\x+1*\dx,\y+0*\dy) -- (\x+1*\dx,\y+1*\dy);
    \node () at (\x+1*\dx,\y+0*\dy) [leg] {.};
    \node () at (\x+1*\dx,\y+1*\dy) [label={[label distance = -2mm]above:\smaller{7}}] {};
    
    \draw[-] (\x+2*\dx,\y+0*\dy) -- (\x+2*\dx,\y+2*\dy);
    \node () at (\x+2*\dx,\y+0*\dy) [leg] {.};
    \draw[-] (\x+4*\dx,\y+0*\dy) -- (\x+4*\dx,\y+2*\dy);
    \node () at (\x+4*\dx,\y+0*\dy) [leg] {.};
    \draw[-] (\x+2*\dx,\y+2*\dy) -- (\x+4*\dx,\y+2*\dy);
    \node () at (\x+3*\dx,\y+2*\dy) [label={[label distance = -2mm]above:\smaller{5}}] {};
    
    \draw[-] (\x+3*\dx,\y+0*\dy) -- (\x+3*\dx,\y+1*\dy);
    \node () at (\x+3*\dx,\y+0*\dy) [leg] {.};
    \node () at (\x+3*\dx,\y+1*\dy) [label={[label distance = -2mm]above:\smaller{7}}] {};
    
    \pgfmathsetmacro {\x}{36*\dx}
    \pgfmathsetmacro {\y}{-6*\dy}

    \draw[-] (\x+0*\dx,\y+0*\dy) -- (\x+0*\dx,\y+3*\dy);
    \node () at (\x+0*\dx,\y+0*\dy) [leg] {.};
    \draw[-] (\x+5*\dx,\y+0*\dy) -- (\x+5*\dx,\y+3*\dy);
    \node () at (\x+5*\dx,\y+0*\dy) [leg] {.};
    \draw[-] (\x+0*\dx,\y+3*\dy) -- (\x+5*\dx,\y+3*\dy);
    \node () at (\x+2.5*\dx,\y+3*\dy) [label={[label distance = -2mm]above:\smaller{2}}] {};
    
    \draw[-] (\x+1*\dx,\y+0*\dy) -- (\x+1*\dx,\y+2*\dy);
    \node () at (\x+1*\dx,\y+0*\dy) [leg] {.};
    \draw[-] (\x+3*\dx,\y+0*\dy) -- (\x+3*\dx,\y+2*\dy);
    \node () at (\x+3*\dx,\y+0*\dy) [leg] {.};
    \draw[-] (\x+1*\dx,\y+2*\dy) -- (\x+3*\dx,\y+2*\dy);
    \node () at (\x+2*\dx,\y+2*\dy) [label={[label distance = -2mm]above:\smaller{7}}] {};
    
    \draw[-] (\x+2*\dx,\y+0*\dy) -- (\x+2*\dx,\y+1*\dy);
    \node () at (\x+2*\dx,\y+0*\dy) [leg] {.};
    \node () at (\x+2*\dx,\y+1*\dy) [label={[label distance = -2mm]above:\smaller{5}}] {};
    
    \draw[-] (\x+4*\dx,\y+0*\dy) -- (\x+4*\dx,\y+1*\dy);
    \node () at (\x+4*\dx,\y+0*\dy) [leg] {.};
    \node () at (\x+4*\dx,\y+1*\dy) [label={[label distance = -2mm]above:\smaller{5}}] {};
    
\end{tikzpicture}
\caption{All labeled non-crossing partitions, adapted to $(2,7,5,7,5,2)$}
\label{figure:adaptedVMonotonePartitions}
\end{figure}
\end{example}

\begin{definition}
\label{definition:combinatorialVmonotoneProduct}
Let us introduce a family $(\psi_n)_{n=1}^{\infty}$ of multilinear functionals such that $\psi_n \colon \bigcup \limits_{(i_1, \ldots, i_n) \in I^n} \mathcal{A}_{i_1} \times \ldots \times \mathcal{A}_{i_n} \mapsto \mathbb{C}$ by the formulas
\begin{equation}
\label{eq:psiEnAsSumOfKappaStarPi}
\psi_n(a_1, \ldots, a_n) = \dsum \limits_{(\pi, \mathfrak{i}) \in \mathcal{V}(i_1, \ldots, i_n)} \kappa_{\pi}^*(a_1, \ldots, a_n) \text{}
\end{equation}
for any $(i_1, \ldots, i_n)$ and any $a_1 \in \mathcal{A}_{i_1}, \ldots, a_n \in \mathcal{A}_{i_n}$, where
\begin{equation}
\label{eq:kappaStarPi}
\kappa_{\pi}^*(a_1, \ldots, a_n) = \dprod \limits_{ \substack{B \in \pi \\ B = \{ p_1 < \ldots < p_q \}} } \kappa_{q}^*(a_{p_1}, \ldots, a_{p_q}) \text{.}
\end{equation}

Let $\mathcal{B} \coloneqq \dbigsqcup \limits_{i \in I} \mathcal{A}_i$ be the free product of algebras without identification of units, let $\sigma_i \colon \mathcal{A}_i \mapsto \mathcal{B}$, $i \in I$, be the family of respective embeddings and let $\psi$ be the functional on $\mathcal{B}$ defined by the linear extension of
\begin{equation}
\label{eq:combinatorialVmonotoneProduct}
\psi(\sigma_{i_1}(a_1) \ldots \sigma_{i_n}(a_n)) = \psi_n(a_1, \ldots, a_n) \text{}
\end{equation}
for any $(i_1, \ldots, i_n) \in I^n$ and any $a_1 \in \mathcal{A}_{i_1}, \ldots, a_n \in \mathcal{A}_{i_n}$.
\end{definition}

Now we have to show that the value of $\psi$ on a simple product $b \in \mathcal{B}$ is independent of the choice of its factorization.

\begin{lemma}
The functional $\psi$ is well defined, i.e. for any $i_1, \ldots, i_n \in I$, $j_1, \ldots, j_m \in I$ and any $a_1 \in \mathcal{A}_{i_1}, \ldots, a_n \in \mathcal{A}_{i_n}$, $b_1 \in \mathcal{A}_{j_1}, \ldots, b_m \in \mathcal{A}_{j_m}$ if
$$
\sigma_{i_1}(a_1) \ldots \sigma_{i_n}(a_n) = \sigma_{j_1}(b_1) \ldots \sigma_{j_m}(b_m) \text{,}
$$
then
$$
\psi(\sigma_{i_1}(a_1) \ldots \sigma_{i_n}(a_n)) = \psi(\sigma_{j_1}(b_1) \ldots \sigma_{j_m}(b_m)) \text{.}
$$
\end{lemma}

\begin{proof}
Let $(i_1, \ldots, i_n) \in I^n$, $a_1 \in \mathcal{A}_{i_1}, \ldots, a_n \in \mathcal{A}_{i_n}$ and $k \in [n]$ be such that $a_k = b_k b'_k$ for some $b_k, b'_k \in \mathcal{A}_{i_k}$. It sufficies to show that
\begin{equation}
\label{eq:combFormulaGenFirstStep}
\dsum \limits_{(\pi, \mathfrak{i}) \in \mathcal{V}(i_1, \ldots, i_n)} \kappa^{*}_{\pi}(a_1, \ldots, a_k, \ldots, a_n) = \dsum \limits_{(\pi, \mathfrak{i}) \in \mathcal{U}} \kappa^{*}_{\pi}(a_1, \ldots, b_k, b'_k, \ldots, a_n) \text{,}
\end{equation}
where $\mathcal{U} = \mathcal{V}(i_1, \ldots, i_{k-1}, i_k, i_k, i_{k+1}, \ldots, i_n)$. First, observe that the set $\mathcal{U}$ can be divided into two disjoint sets of the same cardinality, namely $\mathcal{U} = \mathcal{U}_1 \cup \mathcal{U}_2$. We define $\mathcal{U}_1$ as the set of all labeled partitions $(\pi, \mathfrak{i}) \in \mathcal{U}$ such that there exists a block $B \in \pi$ satisfying $k, k+1 \in B$ and we put $\mathcal{U}_2 = \mathcal{U} \setminus \mathcal{U}_1$. Introduce two functions: $f \colon \mathcal{U}_2 \mapsto \mathcal{U}_1$ and $g \colon \mathcal{U}_1 \mapsto \mathcal{V}(i_1, \ldots, i_n)$. Fix $(\pi, \mathfrak{i}) \in \mathcal{U}_2$ and let $B_1, B_2 \in \pi$ be such that $k \in B_1$ and $k+1 \in B_2$. Note that these blocks have the same label and, since the partition is V-monotonically labeled, $\max(B_1) = k = \min(B_2) - 1$. We define $f(\pi, \mathfrak{i})$ as the labeled partition $\left( (\pi \setminus \{ B_1, B_2 \}) \cup \{ B_1 \cup B_2 \}, \mathfrak{i}' \right)$ such that 
$$
\mathfrak{i}'(B) = 
\begin{cases}
\mathfrak{i}(B) & \text{if $B \neq B_1 \cup B_2$} \\
\mathfrak{i}(B_1) = \mathfrak{i}(B_2) & \text{if $B = B_1 \cup B_2$.}
\end{cases}
$$
Of course, $f$ is a bijection between $\mathcal{U}_2$ and $\mathcal{U}_1$. For $(\pi, \mathfrak{i}) \in \mathcal{U}_1$, we define $g(\pi, \mathfrak{i})$ as the partition in which the legs $k$ and $k+1$ are merged into one, $k$-th leg, and we do not change the label of any block. One can see that $g$ is also a bijection, between $\mathcal{U}_1$ and $\mathcal{V}(i_1, \ldots, i_n)$. 

Let $(\pi, \mathfrak{i}) \in \mathcal{V}(i_1, \ldots, i_n)$ and let $B_0 \in \pi$ of the form $\{ p_1 < \ldots < p_q \}$ be such that $p_r = k$ for some $r \in [q]$. We obtain from Proposition~\ref{proposition:kappaStarFunctionalsProperties} that
\begin{align*}
\kappa^{*}_q(a_{p_1}, \ldots, a_k, \ldots, a_{p_q})
= & \ \kappa^{*}_{q+1}(a_{p_1}, \ldots, b_k, b'_k, \ldots, a_{p_q}) \\
& + \kappa^{*}_k(a_{p_1}, \ldots, b_k) \kappa^{*}_{q+1-k}(b'_k, \ldots, a_{p_q}) \text{,}
\end{align*}
therefore,
$$
\kappa^{*}_{\pi}(a_1, \ldots, a_k, \ldots, a_n) = \kappa^{*}_{\pi'}(a_1, \ldots, b_k, b'_k, \ldots, a_n) + \kappa^{*}_{\pi''}(a_1, \ldots, b_k, b'_k, \ldots, a_n) \text{,}
$$
where $\mathcal{U}_1 \ni (\pi', \mathfrak{i}') = g^{-1}(\pi, \mathfrak{i})$ and $\mathcal{U}_2 \ni (\pi'', \mathfrak{i}'') = (g \circ f)^{-1}(\pi, \mathfrak{i})$, which immediately implies~\eqref{eq:combFormulaGenFirstStep}.

Now, let $a_1 \in \mathcal{A}_{i_1}, \ldots, a_n \in \mathcal{A}_{i_n}$. Assume that 
$$
(i_1, \ldots, i_n) = (\underbrace{j_1, \ldots, j_1}_\text{$n_1$ times}, \ldots, \underbrace{j_m, \ldots, j_m}_\text{$n_m$ times}) \text{,}
$$
where $(j_1, \ldots, j_m) \in I^m$ consists of neighboring different indices. Let $b_1 = a_1 \ldots a_{n_1}, \ldots, b_m = a_{n_1 + \ldots + n_{m-1} + 1} \ldots a_{n_1 + \ldots + n_m}$. We conclude that
$$
\dsum \limits_{(\pi, \mathfrak{i}) \in \mathcal{V}(j_1, \ldots, i_m)} \kappa^{*}_{\pi}(b_1, \ldots, b_m) = \dsum \limits_{(\pi, \mathfrak{i}) \in \mathcal{V}(i_1, \ldots, i_n)} \kappa^{*}_{\pi}(a_1, \ldots, a_n) \text{,}
$$
by formula~\eqref{eq:combFormulaGenFirstStep} and induction, which completes the proof.
\end{proof}

\begin{lemma}
\label{lemma:combFormula}
Assume that $\varphi(\unitInTheAlgebra_i) = 1$ for any $i \in I$. A family $(\mathcal{A}_i)_{i \in I}$ is V-monotone independent with respect to $\varphi$ if and only if it satisfies the condition 
\begin{equation}
\label{eq:combinatorialRuleForMixedMomentsOfVmonotoneRandomVariables}
\varphi(a_1 \ldots a_n) = \psi(\sigma_{i_1}(a_1) \ldots \sigma_{i_n}(a_n)) \text{}
\end{equation}
for any $(i_1, \ldots, i_n) \in I^n$ and any $a_1 \in \mathcal{A}_{i_1}, \ldots, a_n \in \mathcal{A}_{i_n}$.
\end{lemma}

\begin{proof}
We first show the second implication. Let $(i_1, \ldots, i_n)$ be a sequence such that every two neighboring indices are different and let $r \in [n]$ satisfy the condition $I_r \ni (i_1, \ldots, i_r) \nrsim i_{r+1}$ (if $(i_1, \ldots, i_n) \in I_n$, then $r = n$), and let $a_1 \in \mathcal{A}_{i_1}, \ldots, a_n \in \mathcal{A}_{i_n}$. Observe that for every $(\pi, \mathfrak{i}) \in \mathcal{V}(i_1, \ldots, i_n)$ there exists a block $B \in \pi$ which is an interval, because $\pi$ is non-crossing. Moreover, since the labels of neighboring legs are different, $B$ is a singleton block.

\textbf{Case $1^{\circ}$.} Assume that $\varphi(a_1) = \ldots = \varphi(a_n) = 0$. We will show that each summand in RHS of~\eqref{eq:psiEnAsSumOfKappaStarPi} is equal to zero. Let $\pi \in \mathcal{V}(i_1, \ldots, i_n)$ and let $k$ be a singleton leg in $\pi$. In the product $\kappa_{\pi}^*(a_1, \ldots, a_n)$, there exists a factor of the form $\kappa_{1}^*(a_k) = \varphi(a_k) = 0$, and hence $\kappa_{\pi}^*(a_1, \ldots, a_n) = 0$.

\textbf{Case $2^{\circ}$.} Assume that $j \in [n]$ is such that $\varphi(a_1) = \ldots = \varphi(a_j) = 0$ and $a_j = \unitInTheAlgebra_{i_j}$. Consider two cases. 

\textbf{(i)} Let $j > r$. Fix $(\pi, \mathfrak{i}) \in \mathcal{V}(i_1, \ldots, i_n)$. If some $k \in [j-1]$ is a singleton leg, then of course $\kappa_{1}^*(a_k) = 0$ and thus $\kappa_{\pi}^*(a_1, \ldots, a_n) = 0$. 
If there is no singleton leg among $[j-1]$, then every leg in this set is the leftmost leg of its block. On the contrary, suppose that some leg among $[j-1]$ is a middle or the rightmost leg in its block and let $k \in B$ be the smallest such leg. Since neighboring indices in $(i_1, \ldots, i_n)$ are different, there exists a non-crossing partition $\pi'$ on the interval $[\min(B)+1, k-1]$ such that $\pi' \subseteq \pi$, which of course has an interval block. Moreover, it has to be a singleton block, which contradicts our assumption.

In such a case we have $(i_1, \ldots, i_{j-1}) = (\mathfrak{i}(B_1), \ldots \mathfrak{i}(B_{j-1}))$, where $B_l$ is determined by the condition $l = \min(B_l)$, $l \in [j-1]$. However $(i_1, \ldots, i_{j-1}) \notin I_{j-1}$ by assumption, hence $(\pi, \mathfrak{i})$ cannot belong to $\mathcal{V}(i_1, \ldots, i_n)$, a contradiction.

\textbf{(ii)} Let $j \leq r$. We will show that
$$
\dsum \limits_{(\pi, \mathfrak{i}) \in \mathcal{V}(i_1, \ldots, i_n)} \kappa_{\pi}^*(a_1, \ldots, a_n) = \dsum \limits_{(\pi, \mathfrak{i}) \in \mathcal{U}'} \kappa_{\pi}^*(a_1, \ldots, a_{j-1}, a_{j+1}, \ldots, a_n) \text{,}
$$
where $\mathcal{U}' = \mathcal{V}(i_1, \ldots, i_{j-1}, i_{j+1}, \ldots, i_n)$. Let $\mathcal{U}_j$ be the set of all labeled partitions $\pi \in \mathcal{V}(i_1, \ldots, i_n)$ in which $1, \ldots, j-1$ are the leftmost legs in their blocks and $j$ is a singleton leg and let $\mathcal{U}'_j$ be the analogous subset of $\mathcal{V}$, without the $j$-th leg's condition. 

We claim that if $(\pi, \mathfrak{i}) \in \mathcal{V}(i_1, \ldots, i_n) \setminus \mathcal{U}_j$ and $(\pi', \mathfrak{i}') \in \mathcal{U}' \setminus \mathcal{U}'_j$, then $\kappa_{\pi}^*(a_1, \ldots, a_n) = \kappa_{\pi'}^*(a_1, \ldots, a_{j-1}, a_{j+1}, \ldots, a_n) = 0$. Indeed, if we take $(\pi, \mathfrak{i}) \in \mathcal{V}(i_1, \ldots, i_n) \setminus \mathcal{U}_j$, then either some leg among $[j-1]$ is not the leftmost leg in its block and therefore, using similar arguments as before, there is a singleton leg among this set, which implies our assertion (and similarly for $(\pi', \mathfrak{i}') \in \mathcal{U}' \setminus \mathcal{U}'_j$), or $j$ is not a singleton leg and it has to be the leftmost leg in its block, say $B = \{ p_1 < \ldots < p_q \}$. By the definition of $\kappa_{q}^*$, the factor $\kappa_{q}^*(a_{p_1}, \ldots, a_{p_q}) = \kappa_{q}^*(\unitInTheAlgebra_{i_j}, \ldots, a_{p_q}) = 0$ and our assertion is true. We have shown that 
\begin{equation}
\label{eq:indep1}
\varphi(a_1 \ldots a_n) = \dsum \limits_{(\pi, \mathfrak{i}) \in \mathcal{U}_j} \kappa_{\pi}^*(a_1, \ldots, a_n)
\end{equation} 
and
\begin{equation}
\label{eq:indep2}
\varphi(a_1 \ldots a_{j-1} a_{j+1} \ldots a_n) = \dsum \limits_{(\pi', \mathfrak{i}') \in \mathcal{U}'_j} \kappa_{\pi'}^*(a_1, \ldots, a_{j-1}, a_{j+1}, \ldots, a_n) \text{.}
\end{equation}
There exists a bijection between $\mathcal{U}_j$ and $\mathcal{U}'_j$ (the image of $(\pi, \mathfrak{i})$ is the labeled partition, which arises by removing the $j$-th leg from $\pi$). Since $j$ is a singleton leg in each $\pi$ on the RHS of~\eqref{eq:indep1} and $\kappa^*_1(\unitInTheAlgebra_{i_j}) = 1$, the RHS of~\eqref{eq:indep1} is equal to the RHS of~\eqref{eq:indep2} and the proof of this implication is complete. 

Now, we prove the second implication. Observe that we have already shown that the family $(\sigma_i(\mathcal{A}_i))_{i \in I}$ is V-monotone with respect to $\psi$. Let $i_1 \neq \ldots \neq i_n \in I$ and $a_1 \in \mathcal{A}_{i_1}, \ldots, a_n \in \mathcal{A}_{i_n}$. Corollary~\ref{corollary:universalPolynomialFromRecursion} implies that there exists a polynomial $w$ such that
$$
\varphi(a_1 \ldots a_n) = w(\varphi(a_B): B \in \mathcal{I}(i_1, \ldots, i_n))
$$
and
$$
\psi(\sigma_{i_1}(a_1) \ldots \sigma_{i_n}(a_n)) = w(\psi(\tilde{a}_B): B \in \mathcal{I}(i_1, \ldots, i_n)) \text{,}
$$
where $\tilde{a}_B = \sigma_{j}(a_{p_1}) \ldots \sigma_{j}(a_{p_q})$ for $B = \{ p_1 < \ldots < p_q \}$ and $j \coloneqq i_{p_1} = \ldots = i_{p_q}$. Since $\sigma_j$ is a homomorphism, by the definition of $\psi$, we have $(\psi \circ \sigma_j)(a_B) = (\psi_1 \circ \sigma_j)(a_B) = \varphi(a_B)$, which gives~\eqref{eq:combinatorialRuleForMixedMomentsOfVmonotoneRandomVariables}.
\end{proof}

\begin{corollary}
\label{corollary:combFormula}
If the family $(\mathcal{A}_i)_{i \in I}$ is V-monotone independent with respect to $\varphi$, then for any $i_1, \ldots, i_n \in I$ and any $a_1 \in \mathcal{A}_{i_1}, \ldots, a_n \in \mathcal{A}_{i_n}$ we have
\begin{equation}
\label{eq:combFormula}
\varphi(a_1 \ldots a_n) = \dsum \limits_{(\pi, \mathfrak{i}) \in \mathcal{V}(i_1, \ldots, i_n)} \kappa_{\pi}^*(a_1, \ldots, a_n) \text{.}
\end{equation}
\end{corollary}

\begin{remark}
Formulas analogous to~\eqref{eq:combFormula} can be proven in the free and monotone cases, but one has to replace $\mathcal{V}(i_1, \ldots, i_n)$ with other classes of labeled partitions: one replaces the V-monotone labeling with free labeling (each block has a different label than its nearest outer block) and monotone labeling (each block has a greater label than its nearest outer block), respectively.
\end{remark}

\section{Hilbert space realization}
\label{section:hilbertSpaceRealization}
Consider the family of $C^{*}$-probability spaces $(\mathcal{A}_i, \varphi_i)$ with units $(\unitInTheAlgebra_i)_{i \in I}$ (i.e. additionally for each $i \in I$ and $a \in \mathcal{A}_i$ we have $\varphi_i(a^{*} a) \geq 0$) and the family $(\mathcal{H}_i, \sigma_i,  \xi_i)_{i \in I}$ of the associated GNS triples (that is for any $i \in I$ and any $a \in \mathcal{A}_i$ we have $\varphi_i(a) = \langle \sigma_i(a) \, \xi_i, \xi_i \rangle$). Let $P_i$ be the orthogonal projection onto $\mathbb{C} \, \xi_i$ and $P_i^\perp = \id_i - P_i$. For $T \in \bound(\mathcal{H}_i)$ we define $T^{\perp} \coloneqq P_i^{\perp} T$. Let $\varphi_i(T) = \langle T \, \xi_i, \xi_i \rangle$. Now we recall the definitions of the Voiculescu's free product of $(\mathcal{H}_i, \xi_i)$ and the left representation of $\bound(\mathcal{H}_i)$ on it, namely $\lambda_i$ (see for instance~\cite{VoiculescuDykemaNica1992}).

By the the Voiculescu's free product of $(\mathcal{H}_i, \xi_i)$ we understand the Hilbert space given by
$$
\mathcal{F} \coloneqq \mathbb{C} \, \xi \oplus \dirsum{n=1}{\infty} \dirsum{i_1 \neq \ldots \neq i_n}{} \ovcirc{H}_{i_1} \otimes \ldots \otimes \ovcirc{H}_{i_n} \text{}
$$
with the cannonical scalar product, where $\ovcirc{H}_i = \mathcal{H}_i \ominus \mathbb{C} \, \xi_i$ and $\xi$ is a unit vector. Introduce the vacuum state $\varphi$ on $\bound(\mathcal{F})$, namely for $T \in \bound(\mathcal{F})$ we define $\varphi(T) = \langle T \, \xi, \xi \rangle$. For each $i \in I$, we introduce an auxiliary subspace, namely
$$
\mathcal{F}_i \coloneqq \mathbb{C} \, \xi \oplus \dirsum{n=1}{\infty} \dirsum{\substack{i_1 \neq \ldots \neq i_n \\ i \neq i_1}}{} \ovcirc{H}_{i_1} \otimes \ldots \otimes \ovcirc{H}_{i_n} \text{}
$$
and the unitary operator $V_i \colon \mathcal{H}_i \otimes \mathcal{F}_i \mapsto \mathcal{F}$ given by the continuous linear extension of
\begin{align*}
& V_i \, \xi_i \otimes \xi = \xi \\
& V_i \, x \otimes \xi = x \\
& V_i \, \xi_i \otimes (x_1 \otimes \ldots \otimes x_n) = x_1 \otimes \ldots \otimes x_n \\
& V_i \, x \otimes (x_1 \otimes \ldots \otimes x_n) = x \otimes x_1 \otimes \ldots \otimes x_n \text{,}
\end{align*}
where $x \in \ovcirc{H}_i, x_1 \in \ovcirc{H}_{i_1}, \ldots, x_n \in \ovcirc{H}_{i_n}$ for $i \neq i_1 \neq \ldots \neq i_n$. We put $\lambda_i(T) \coloneqq V_i (T \otimes \id) V_i^{-1}$.

The Hilbert space defined below is an analogue of Muraki's monotone product of Hilbert spaces (see~\cite{Muraki2000}).

\begin{definition}
The \emph{V-monotone product} of Hilbert spaces $(\mathcal{H}_i, \xi_i)_{i \in I}$ is the Hilbert subspace of $\mathcal{F}$ given by
$$
\mathcal{V} \coloneqq \mathbb{C} \, \xi \oplus \dirsum{n=1}{\infty} \dirsum{(i_1, \ldots, i_n) \in I_n}{} \ovcirc{H}_{i_1} \otimes \ldots \otimes \ovcirc{H}_{i_n} \text{,}
$$
where $I_n$ were introduced in Definition~\ref{definition:I_n-indices}. For any $i \in I$, we introduce the \emph{V-monotone left representation} of $\bound(\mathcal{H}_i)$ on $\mathcal{V}$ by
$$
\widetilde{\lambda}_i(T) = U_i \lambda_i(T) U_i \text{,}
$$
where $T \in \bound(\mathcal{H}_i)$, and $U_i$ is the orthogonal projection onto 
$$
\mathcal{V}_i \coloneqq \mathbb{C} \, \xi \oplus \dirsum{n=1}{\infty} \dirsum{(i_1, \ldots, i_n) \in I_n(i)}{} \ovcirc{H}_{i_1} \otimes \ldots \otimes \ovcirc{H}_{i_n} \text{}
$$
for $I_n(i) = \{ (i_1, \ldots, i_n) \in I_n: i_1 = i \text{ or } i \lsim (i_1, \ldots, i_n) \}$.
\end{definition}

\begin{proposition}
\label{proposition:actingOfLambdaTilde}
Let $x \in \ovcirc{H}_i$, $x_k \in \ovcirc{H}_{i_k}$ and $y_k \in \ovcirc{H}_{j_k}$, $k \in [n]$, for indices such that $i \lsim (i_1, \ldots, i_n) \in I_n$ and $(j_1, \ldots, j_n) \in I_n \setminus I_n(i)$. Then, for $T \in \bound(\mathcal{H}_i)$, we have
\begin{equation*}
\begin{split}
& \widetilde{\lambda}_i(T) \, \xi = \varphi_i(T) \xi + T^{\perp} \, \xi_i \text{,} \\
& \widetilde{\lambda}_i(T) \, x = \langle T \, x, \xi_i \rangle \xi + T^{\perp} \, x \text{,} \\
& \widetilde{\lambda}_i(T) \, x \otimes x_1 \otimes \ldots \otimes x_n = \langle T \, x, \xi_i \rangle x_1 \otimes \ldots \otimes x_n + T^{\perp} \, x \otimes x_1 \otimes \ldots \otimes x_n \text{,} \\
& \widetilde{\lambda}_i(T) \, x_1 \otimes \ldots \otimes x_n = \varphi_i(T) x_1 \otimes \ldots \otimes x_n + T^{\perp} \, \xi_i \otimes x_1 \otimes \ldots \otimes x_n \text{,} \\
& \widetilde{\lambda}_i(T) \, y_1 \otimes \ldots \otimes y_n = 0 \text{.}
\end{split}
\end{equation*}
Moreover, $\widetilde{\lambda}_i(\unitInTheAlgebra_i) = U_i$.
\end{proposition}

\begin{proof}
We will prove for instance, the second and fourth equalities. We have
\begin{align*}
\widetilde{\lambda}_i(T) \, x & = U_i V_i (T \otimes \id) \, x \otimes \xi = U_i V_i \, (\langle T \, x, \xi_i \rangle \xi_i \otimes \xi + (T^{\perp} \, x) \otimes \xi) \\
& = U_i \, (\langle T \, x, \xi_i \rangle \xi + T^{\perp} \, x) = \langle T \, x, \xi_i \rangle \xi + T^{\perp} \, x
\end{align*}
and for $v = x_1 \otimes \ldots \otimes x_n$
\begin{align*}
\widetilde{\lambda}_i(T) \, v & = U_i V_i (T \otimes \id) \, \xi_i \otimes v = U_i V_i (\varphi_i(T) \xi_i \otimes v + (T^{\perp} \, \xi_i) \otimes v) \\
& = U_i (\varphi_i(T) v + (T^{\perp} \, \xi_i) \otimes v) = \varphi_i(T) v + (T^{\perp} \, \xi_i) \otimes v \text{.}
\end{align*}
\end{proof}

\begin{proposition}
The operator $\widetilde{\lambda}_i$ is a non-unital $*$-representation.
\end{proposition}

\begin{proof}
Let $T, S \in \bound(\mathcal{H}_i)$. Since $\lambda_i(T) \, \mathcal{V}_i \subseteq \mathcal{V}_i$ and $U_i$ is an orthogonal projection onto $\mathcal{V}_i$, we get $U_i \lambda_i(T) U_i \lambda_i(S) U_i = U_i \lambda_i(T) \lambda_i(S) U_i$. The fact that $\lambda_i$ is a $*$-homomorphism finishes the proof.
\end{proof}

\begin{definition}
\label{definition:definitionOfVmonotoneProductOfStates}
We define the \emph{V-monotone product of states} on $\mathcal{A} = \dbigsqcup \limits_{i \in I} \mathcal{A}_i$ (the free product of $C^{*}$-algebras without identification of units) as the state $\phi$ given by the continuous linear extension of
\begin{equation*}
\begin{split}
\phi(a_1 \ldots a_n) = \varphi( (\widetilde{\lambda}_{i_1} \circ \sigma_{i_1})(a_1) \ldots (\widetilde{\lambda}_{i_n} \circ \sigma_{i_n})(a_n) )
\end{split}
\end{equation*}
for any $(i_1, \ldots, i_n) \in I^n$ and any $a_1 \in \mathcal{A}_{i_1}, \ldots, a_n \in \mathcal{A}_{i_n}$. We introduce the notation:
$$
\vProduct \limits_{i \in I} \varphi_i \coloneqq \phi \qquad \text{and} \qquad \vProduct \limits_{i \in I} (\mathcal{A}_i, \varphi_i) \coloneqq (\mathcal{A}, \phi) \text{.}
$$
\end{definition}

\begin{theorem}
\label{thm:realization}
The family $( (\widetilde{\lambda}_i \circ \sigma_i)(\mathcal{A}_i))_{i \in I}$ is V-monotone independent with respect to $\varphi$.
\end{theorem}

\begin{proof}
Let $T_l = \sigma_{i_l}(a_l)$, $l \in [n]$. First, note that if $T_1 \in \ker \varphi_{i_1}, \ldots, T_l \in \ker \varphi_{i_l}$ for $i_1, \ldots, i_l \in I$ such that the neighboring indices are different, then by induction and Proposition~\ref{proposition:actingOfLambdaTilde}
\begin{equation}
\label{eq:representation1}
\widetilde{\lambda}_{i_1}(T_1) \ldots \widetilde{\lambda}_{i_l}(T_l) \, \xi =
\begin{cases}
T_1^{\perp} \xi_{i_1} \otimes \ldots \otimes T_l^{\perp} \xi_{i_l} & \text{if $(i_1, \ldots, i_l) \in I_l$} \\
0 & \text{otherwise.}
\end{cases}
\end{equation}
It is clear then that the freeness condition is fulfilled. Let $i_1 \neq \ldots \neq i_n$, $r \in [n]$ be such that $I_r \ni (i_1, \ldots, i_r) \nrsim i_{r+1}$ and let $j \in [n]$ be such that $T_1 \in \ker \varphi_{i_1}, \ldots, T_{j-1} \in \ker \varphi_{i_{j-1}}$ and $T_j = \unitInTheAlgebra_j$. Consider the term 
$$
\varphi \left( \widetilde{\lambda}_{i_n}(T^{*}_n) \ldots \widetilde{\lambda}_{i_{j+1}}(T^{*}_{j+1}) U_{i_j} \widetilde{\lambda}_{i_{j-1}}(T^{*}_{j-1}) \ldots \widetilde{\lambda}_{i_1}(T^{*}_1) \right) \text{.}
$$
Let $h = (T_{j-1}^{*})^{\perp} \xi_{i_{j-1}} \otimes \ldots \otimes (T_1^{*})^{\perp} \xi_{i_1}$. If $j > r$ then either $j > r+1$ and~\eqref{eq:representation1} gives our assertion, or $j = r+1$ and $U_{i_{r+1}} \, h = 0$. If $j \leq r$, then $U_{i_j} \, h = h$. The hermiticity of $\varphi$ finishes the proof.
\end{proof}

\begin{proposition}
\label{proposition:notAssociative}
The V-monotone product of states is not associative.
\end{proposition}

\begin{proof}
Let $(\mathcal{A}_i, \varphi_i)$, for $i=1,2,3$, be $C^{*}$-probability spaces and let $a_1, a_2 \in \mathcal{A}_1$, $b \in \mathcal{A}_1$ and $c_1, c_2 \in \mathcal{A}_3$. Let $\phi_{12} = \varphi_1 \ovee \varphi_2$ and $\phi_{23} = \varphi_2 \ovee \varphi_3$. The formulas in Example~\ref{example:firstMixedMomentsForTwoAlgebras} yield
\begin{align*}
( & \phi_{12} \ovee \varphi_3)(a_1 c_1 b c_2 a_2) = \phi_{12}(a_1 b a_2) \varphi_3(c_1) \varphi_3(c_2) \\
& + \phi_{12}(a_1) \phi_{12}(b) \phi_{12}(a_2) \varphi_3(c_1 c_2) - \phi_{12}(a_1) \phi_{12}(b) \phi_{12}(a_2) \varphi_3(c_1) \varphi_3(c_2) \\
= & \ \varphi_1(a_1 a_2) \varphi_2(b) \varphi_3(c_1) \varphi_3(c_2) + \varphi_1(a_1) \varphi_1(a_2) \varphi_2(b) \varphi_3(c_1 c_2) \\
& - \varphi_1(a_1) \varphi_1(a_2) \varphi_2(b) \varphi_3(c_1) \varphi_3(c_2)
\end{align*}
and
\begin{multline*}
(\varphi_1 \ovee \phi_{23})(a_1 c_1 b c_2 a_2) = \varphi_1(a_1 a_2) \phi_{23}(c_1 b c_2) = \varphi_1(a_1 a_2) \varphi_2(b) \varphi_3(c_1 c_2) \text{.}
\end{multline*}
These two mixed moments are different, which completes the proof.
\end{proof}

\section{Central limit theorem}
\label{section:centralLimitTheorem}
In Sections~\ref{section:centralLimitTheorem} and~\ref{section:VmonotoneGaussianOperators}, we assume that $I = \mathbb{N}_{+}$ with the natural order. In this section, we state and prove the V-monotone central limit theorem. We express the limit moments in terms of combinatorial objects and we find a recurrence formula for the even moments (the odd moments vanish). Let $\mathcal{ONC}(n)$ be the set of all ordered partitions on $[n]$ which are non-crossing (see Definition~\ref{definition:orderedPartitions}).

\begin{definition}
Let $\mathcal{OV}(n)$ be the set of all $(\pi, \mathfrak{i}) \in \mathcal{ONC}(n)$ which are V-monotonically labeled. We also denote by $\mathcal{OV}^2(2n)$ the set of all $(\pi, \mathfrak{i}) \in \mathcal{OV}(2n)$ such that $\pi$ is a pair partition.
\end{definition}

\begin{example}
For $n<5$, we have $\mathcal{OV}(n) = \mathcal{ONC}(n)$. The set $\mathcal{ONC}(5) \setminus \mathcal{OV}(5)$ consists of two elements, shown in Fig.~\ref{figure:orderedNotVMonotonePartitions}.

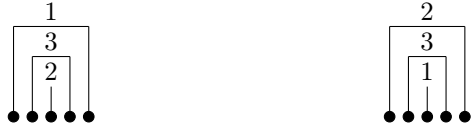
\begin{figure}[H]
\centering
\begin{tikzpicture}
	\tikzstyle{leg} = [circle, draw=black, fill=black!100, text=black!100, thin, inner sep=0pt, minimum size=4.0]

    \pgfmathsetmacro {\dx}{0.25}
    \pgfmathsetmacro {\dy}{0.40}

    \pgfmathsetmacro {\x}{0}
    \pgfmathsetmacro {\y}{0}

    \draw[-] (\x+0*\dx,\y+0*\dy) -- (\x+0*\dx,\y+3*\dy);
    \node () at (\x+0*\dx,\y+0*\dy) [leg] {.};
    \draw[-] (\x+4*\dx,\y+0*\dy) -- (\x+4*\dx,\y+3*\dy);
    \node () at (\x+4*\dx,\y+0*\dy) [leg] {.};
    \draw[-] (\x+0*\dx,\y+3*\dy) -- (\x+4*\dx,\y+3*\dy);
    \node () at (\x+2*\dx,\y+3*\dy) [label={[label distance = -2mm]above:\smaller{1}}] {};

    \draw[-] (\x+1*\dx,\y+0*\dy) -- (\x+1*\dx,\y+2*\dy);
    \node () at (\x+1*\dx,\y+0*\dy) [leg] {.};
    \draw[-] (\x+3*\dx,\y+0*\dy) -- (\x+3*\dx,\y+2*\dy);
    \node () at (\x+3*\dx,\y+0*\dy) [leg] {.};
    \draw[-] (\x+1*\dx,\y+2*\dy) -- (\x+3*\dx,\y+2*\dy);
    \node () at (\x+2*\dx,\y+2*\dy) [label={[label distance = -2mm]above:\smaller{3}}] {};
    
    \draw[-] (\x+2*\dx,\y+0*\dy) -- (\x+2*\dx,\y+1*\dy);
    \node () at (\x+2*\dx,\y+0*\dy) [leg] {.};
    \node () at (\x+2*\dx,\y+1*\dy) [label={[label distance = -2mm]above:\smaller{2}}] {};
	
    \pgfmathsetmacro {\x}{5}
    \pgfmathsetmacro {\y}{0}

    \draw[-] (\x+0*\dx,\y+0*\dy) -- (\x+0*\dx,\y+3*\dy);
    \node () at (\x+0*\dx,\y+0*\dy) [leg] {.};
    \draw[-] (\x+4*\dx,\y+0*\dy) -- (\x+4*\dx,\y+3*\dy);
    \node () at (\x+4*\dx,\y+0*\dy) [leg] {.};
    \draw[-] (\x+0*\dx,\y+3*\dy) -- (\x+4*\dx,\y+3*\dy);
    \node () at (\x+2*\dx,\y+3*\dy) [label={[label distance = -2mm]above:\smaller{2}}] {};

    \draw[-] (\x+1*\dx,\y+0*\dy) -- (\x+1*\dx,\y+2*\dy);
    \node () at (\x+1*\dx,\y+0*\dy) [leg] {.};
    \draw[-] (\x+3*\dx,\y+0*\dy) -- (\x+3*\dx,\y+2*\dy);
    \node () at (\x+3*\dx,\y+0*\dy) [leg] {.};
    \draw[-] (\x+1*\dx,\y+2*\dy) -- (\x+3*\dx,\y+2*\dy);
    \node () at (\x+2*\dx,\y+2*\dy) [label={[label distance = -2mm]above:\smaller{3}}] {};
    
    \draw[-] (\x+2*\dx,\y+0*\dy) -- (\x+2*\dx,\y+1*\dy);
    \node () at (\x+2*\dx,\y+0*\dy) [leg] {.};
    \node () at (\x+2*\dx,\y+1*\dy) [label={[label distance = -2mm]above:\smaller{1}}] {};
	
	\node () at (0, -0.25) {};
\end{tikzpicture}
\caption{Ordered non-crossing partitions.}
\label{figure:orderedNotVMonotonePartitions}
\end{figure}
\end{example}

\begin{theorem}
\label{thm:CLT}
Let $(\mathcal{A}, \varphi)$ be a $C^{*}$-probability space and let $(a_i)_{i=1}^{\infty}$ be a family of V-monotonically independent (with respect to $\varphi$) and identically distributed random variables with mean $0$ and variance $1$. Let
$$
S_N = \dfrac{1}{\sqrt{N}} \sum \limits_{i=1}^N a_i \text{.}
$$
Then, for all $n \in \mathbb{N}$, we have
$$
\lim \limits_{N \to \infty} \varphi(S_N^n) = 
\begin{cases}
\dfrac{ \card{ \mathcal{OV}^2(2k) } }{k!} & \text{if $n=2k$, $k \in \mathbb{N}$} \\
0 & \text{if $n$ is odd.}
\end{cases}
$$
The distribution given by the above moment sequence will be called the \emph{standard V-monotone Gaussian distribution}.
\end{theorem}

\begin{proof}
Let $(\mu_n)_{n=0}^{\infty}$ be the moment sequence of the random variable $a_i$, $i \in \mathbb{N}_{+}$ and let $\kappa^{*}_n = \kappa_{n}^*(a_i,\ldots,a_i)$, $n \in \mathbb{N}_{+}$ (it does not depend on $i$ since the variables are identically distributed). Of course, we have $\kappa^{*}_1 = 0$, $\kappa^{*}_2 = 1$, and $\kappa^{*}_n$ depends only on $\mu_1, \ldots, \mu_n$. Fix $N \in \mathbb{N}_{+}$ and $n \in \mathbb{N}$. Corollary~\ref{corollary:combFormula} implies that
\begin{align*}
\varphi(S_N^n) 
& = N^{-n/2} \sum \limits_{(i_1, \ldots, i_n) \in [N]^n} \varphi(a_{i_1} \ldots a_{i_n}) \\
& = N^{-n/2} \sum \limits_{(i_1, \ldots, i_n) \in [N]^n} \sum \limits_{(\pi, \mathfrak{i}) \in \mathcal{V}(i_1, \ldots, i_n)} \kappa^{*}_{\pi}(a_{i_1}, \ldots, a_{i_n}) \text{.}
\end{align*}
We have $\kappa^{*}_{r}(a_{i_{p_1}}, \ldots, a_{i_{p_r}}) = \kappa^{*}_r$, since $i_{p_1} = \ldots = i_{p_r}$ in each factor in the above product, and therefore, we denote
$$
\kappa^{*}_{\pi} \coloneqq \kappa_{\pi}^*(a_{i_1}, \ldots, a_{i_n}) = \prod \limits_{ B \in \pi } \kappa^{*}_{ \card{B} } \text{,}
$$
and thus
$$
\varphi(S_N^n) = N^{-n/2} \sum \limits_{(i_1, \ldots, i_n) \in [N]^n} \sum \limits_{(\pi, \mathfrak{i}) \in \mathcal{V}(i_1, \ldots, i_n)} \kappa^{*}_{\pi} \text{.}
$$
If some index $i \in [N]$ appears in $(i_1, \ldots, i_n)$ exactly once, then for any $(\pi, \mathfrak{i}) \in \mathcal{V}(i_1, \ldots, i_n)$ we have $\mu_{\pi} = 0$, because $\pi$ must have a singleton block labeled with $i$ and $\kappa^{*}_1 = 0$. There are at most $O(N^{(n-1)/2})$ sequences $(i_1, \ldots, i_n) \in [N]^n$ in which there are no singleton indices and which have at least one index $i \in \{ i_1, \ldots, i_n \}$ which appears in $(i_1, \ldots, i_n)$ at least three times. Indeed, we first choose a partition $\pi \in \mathcal{P}(n)$ which has at least one block with at least three legs and the remaining blocks have at least two legs, choosen in $O(1)$ ways. Clearly, it has at most $1+\tfrac{n-3}{2} = \tfrac{n-1}{2}$ blocks. We now choose labels of these blocks in at most $N^{(n-1)/2}$ ways. This labeled partition determines the sequence $(i_1, \ldots, i_n)$ which has the desired property.

Let $\mathcal{I}_n$ be the set of all sequences $(i_1, \ldots, i_n) \in [N]^n$ in which each index appears exactly twice. Then
$$
\varphi(S_N^n) = N^{-n/2} \sum \limits_{(i_1, \ldots, i_n) \in \mathcal{I}_n} \sum \limits_{(\pi, \mathfrak{i}) \in \mathcal{V}(i_1, \ldots, i_n)} \kappa^{*}_{\pi} + O \left( \tfrac{1}{\sqrt{N}} \right) \text{.}
$$
If $n$ is odd, then of course
$$
\lim \limits_{N \to \infty} \varphi(S_N^n) = 0 \text{.}
$$
Let us introduce an equivalence relation on $\mathcal{I}_{2n}$. Namely, two sequences of numbers $(i_1, \ldots, i_{2n})$ and $(j_1, \ldots, j_{2n})$ are equivalent if for any $p, q \in [2n]$, the following implications hold:
\begin{itemize}
	\item[ (i)] $i_p = i_q \Rightarrow j_p = j_q$
	\item[(ii)] $i_p < i_q \Rightarrow j_p < j_q$.
\end{itemize}
Each equivalence class has ${N \choose n}$ elements and the quotient set $\mathcal{J}_{2n}$ has the same cardinality as $\mathcal{OP}^2(2n)$. We construct a bijection between $\mathcal{J}_{2n}$ and $\mathcal{OP}^2(2n)$ as follows: we assign to the equivalence class of $(i_1, \ldots, i_{2n})$ the ordered partition $(\pi, \mathfrak{i})$ such that $\pi$ satisfies the condition: $k, l \in [2n]$ are in one block $B \in \pi$ if and only if $i_k = i_l$, then we put $\mathfrak{i}(B) = r$ such that $i_k$ is the $r$-th element in $\{ i_1, \ldots, i_{2n} \}$ with respect to the natural order. It is clear that for two equivalent sequences the respective ordered partitions are the same. The reader can check that this is indeed a bijection.

Observe that if $(i_1, \ldots, i_{2n})$ and $(j_1, \ldots, j_{2n})$ are equivalent, then for any partition $\pi \in \mathcal{P}(2n)$ there exists a unique labeling $\mathfrak{i}$ such that $(i_1, \ldots, i_{2n})$ is adapted to $(\pi, \mathfrak{i})$ if and only if there exists a unique labeling $\mathfrak{j}$ such that $(\pi, \mathfrak{j})$ is adapted to $(j_1, \ldots, j_{2n})$. Moreover, $(\pi, \mathfrak{i}) \in \mathcal{V}(i_1, \ldots, i_{2n})$ if and only if $(\pi, \mathfrak{j}) \in \mathcal{V}(j_1, \ldots, j_{2n})$. Therefore,
\begin{align*}
\varphi(S_N^{2n}) 
& = N^{-n} \sum \limits_{(i_1, \ldots, i_n) \in \mathcal{J}_{n}} {N \choose n} \sum \limits_{(\pi, \mathfrak{i}) \in \mathcal{V}(i_1, \ldots, i_n)} \kappa^{*}_{\pi} + O(N^{-1/2}) \\
& = N^{-n} {N \choose n} \sum \limits_{(\pi, \mathfrak{i}) \in \mathcal{OV}^2(2n)} \kappa^{*}_{\pi} + O(N^{-1/2}) \\
& = N^{-n} {N \choose n} \card{ \mathcal{OV}^2(2n) } + O(N^{-1/2}) \text{,}
\end{align*}
since $\kappa^{*}_{\pi} = 1$ for any pair partition. Letting $N \to \infty$, we obtain the desired formula.
\end{proof}

Now, we find a recurrence for the cardinalities of $\mathcal{OV}^2(2n)$. For that purpose, we need some auxiliary class of labeled partitions which admits one additional label, for convenience chosen to be a fraction. For each $k \in [n+1]$ we denote by $\mathcal{OV}_k^2(2n)$ the set of all $(\pi, \mathfrak{i}) \in \mathcal{OV}^2(2n)$ such that if we add the block $\{ 0,2n+1 \}$ to $\pi$, labeled with $k-\tfrac{1}{2}$, the partition remains V-monotonically labeled. By $N_{n,k}$ we denote its cardinality and we put $N_{0,1} = 1$. Clearly, $N_{n,n+1}$ and $N_{n,1}$ are the numbers of all ordered pair partitions on $[2n]$ which are V-monotonically and monotonically labeled (i.e. if $B'$ is inner with respect to $B$, then its label is greater than the label of $B$), respectively.

\begin{lemma}
\label{lemma:discreteRecurrenceCLT}
For any natural numbers $n$ and $k \in [n+2]$, we have
$$
N_{n+1, k} = \sum \limits_{m=0}^n \sum \limits_{l = L_1(k,m)}^{L_2(k,m)} { {k-1} \choose {l} }{ {n+2-k} \choose {m+1-l} } \left[ \delta_{l,0} N_{m,1} + \sum \limits_{r=1}^l N_{m,r} \right] N_{n-m,k-l} \text{,}
$$
where $L_1(k,m) = \max(0, (m+k)-(n+1))$, $L_2(k,m) = \min(k-1,m+1)$ and $\delta_{l,0}$ is the Kronecker delta. 
\end{lemma}

\begin{proof}
Fix $n$ and $m \in [0,n]$. Let us consider two disjoint subsets $C, D \subseteq [n+1]$ of the form $\{ i_1, \ldots, i_{m+1} \}$ and $\{ j_1, \ldots, j_{n-m} \}$, respectively, satisfying $C \cup D = [n+1]$ and such that for some non-negative integers $l, l'$ we have
\begin{equation}
\label{eq:labelsInDiscreteRecurrence}
\begin{gathered}
i_1 < \ldots < i_l < k - \tfrac{1}{2} < i_{l+1} < \ldots < i_{m+1} \text{,} \\
j_1 < \ldots < j_{l'} < k - \tfrac{1}{2} < j_{l'+1} < \ldots < j_{n-m} \text{.}
\end{gathered}
\end{equation}
Note that the following conditions hold:
\begin{itemize}
	\item $l + l' = k-1$,
	\item $0 \leq l \leq k-1$,
	\item $0 \leq l' \leq k-1$,
	\item $0 \leq m+1-l \leq n-k+1$,
	\item $0 \leq n-m-l' \leq n-k+1$.
\end{itemize}
The reader can check that these conditions imply $L_1(k,m) \leq l \leq L_2(k,m)$. 

We count in how many ways we can choose $(\pi, \mathfrak{i}) \in \mathcal{OV}^2_k(2n+2)$ for $k \in [n+2]$ (a partition $\pi$ has the form $\pairPartitionRecurrence{\pi'}{\pi''}$ for some non-crossing pair partitions $\pi'$ and $\pi''$). First, we choose $m \in \{ 0, \ldots, n \}$ such that $B = \{1,2m+2\} \in \pi$. Since the fact that $\pi$ is $V$-monotonically labeled is equivalent to the same fact for $\pi' \cup \{ B \}$ and $\pi''$ simultaneously, we shall deal with these two sets of blocks separately. 

Now, we choose $l$ such that $L_1(k,m) \leq l \leq L_2(k,m)$ and labels $i_1 < \ldots < i_l$ from $[k-1]$ (in one of ${ {k-1} \choose {l} }$ ways) and, independently, $i_{l+1} < \ldots < i_{m+1} \in [k,n+1]$ (in one of ${ {n+2-k} \choose {m+1-l} }$ ways). The labels $i_1, \ldots, i_{m+1}$ will be later assigned to blocks of $\pi' \cup \{ B \}$ and the remaining labels $j_1 < \ldots < j_{n-m} \in [n+1]$ will be assigned to blocks of $\pi''$. Moreover, the condition~\eqref{eq:labelsInDiscreteRecurrence} is fulfilled (with $l' = k-1-l$). We emphasize that the partitions $\pi'$ and $\pi''$ have not been chosen yet. At this point we have only determined how many labels of types $i$ and $j$ are less than $k-\tfrac{1}{2}$, respectively.

Next, we choose $r \in [m+1]$ such that $\mathfrak{i}(B) = i_r$. There are two possibilities. If  $l=0$, then $i_r > k-\tfrac{1}{2}$, and hence $\pi'$ must be labeled monotonically. This implies $r=1$. In that case, there are $N_{m,1}$ possibilities of constructing $\pi'$ and labeling its blocks with labels of type $i$. If $l > 0$, then we cannot choose $r > l$, because in that case $\pi'$ would not be labeled V-monotonically (because a local maximum occurs in the index sequence $(k-\tfrac{1}{2}, i_r, i_1)$, which corresponds to the chain of blocks $\{ \{0, 2n+3 \} < B < B' \}$, where $B' \in \pi'$ satisfies $\mathfrak{i}(B') = i_1$). Thus, we take $r \in [l]$ and a partition $\pi'$ labeled with $i_1, \ldots, i_{r-1}, i_{r+1}, \ldots, i_{m+1}$. There are $N_{m,r}$ ways of doing it (observe that this number is independent of the exact values of $i$'s; only the value of $m$ and the position of $\mathfrak{i}(B)$ in the sequence $(i_1, \ldots, i_{m+1})$ are relevant here). Finally, we choose a partition $\pi''$ which is V-monotonically labeled with $j$'s. Since $k - \tfrac{1}{2}$ is the $(k-l)$-th element in the sequence $(j_1, \ldots, j_{l'}, k - \tfrac{1}{2}, j_{l'}, \ldots, j_{n-m})$, there are $N_{n-m,k-l}$ possibilities of doing it. This completes the proof.
\end{proof}

\begin{example}
Below, we give even moments of the lowest order obtained in the V-monotone central limit theorem (second row) and, for comparison, the moments of the standard arcisne law (third row). In the first row we put the order of the moment.
\begin{center}
\renewcommand{\arraystretch}{1.5}
\begin{tabular}{ |c|c|c|c|c|c|c|c|c|c| } 
 \hline
 $2$ & $4$ & $6$ & $8$ & $10$ & $12$ & $14$ & $16$ & $18$ & $20$ \\ 
 \hline
 $\tfrac{1}{1!}$ & $\tfrac{4}{2!}$  & $\tfrac{28}{3!}$ & $\tfrac{278}{4!}$ & $\tfrac{3 \, 564}{5!}$ & $\tfrac{55 \, 928}{6!}$ & $\tfrac{1 \, 037 \, 708}{7!}$ & $\tfrac{22 \, 217 \, 720}{8!}$ & $\tfrac{539 \, 070 \, 560}{9!}$ & $\tfrac{1 \, 731 \, 430 \, 024}{10!}$ \\ 
 \hline
 $\tfrac{1}{1!}$ & $\tfrac{3}{2!}$  & $\tfrac{15}{3!}$ & $\tfrac{105}{4!}$ & $\tfrac{945}{5!}$ & $\tfrac{10 \, 395}{6!}$ & $\tfrac{135 \, 135}{7!}$ & $\tfrac{2 \, 027 \, 025}{8!}$ & $\tfrac{34 \, 459 \, 425}{9!}$ & $\tfrac{654 \, 729 \, 075}{10!}$ \\ 
 \hline
\end{tabular}
\renewcommand{\arraystretch}{1}
\end{center}
The sequence of numerators in the second row cannot be found in \emph{On-line Encyclopedia of Integer Sequences}${}^\circledR$ (\texttt{https://oeis.org/}).
\end{example}

\section{V-monotone Gaussian operators}
\label{section:VmonotoneGaussianOperators}
It is difficult to obtain the moment generating function for the central limit distribution from Lemma~\ref{lemma:discreteRecurrenceCLT}. Below, we present another approach based on constructing an operator which has the standard V-monotone Gaussian distribution. 

\begin{definition}
For $1 \leq m \leq n$ let $\mathcal{X}_{n,m} \coloneqq \{ (x_1, \ldots, x_n) \in \mathbb{R}^n: x_1 > \ldots > x_m < \ldots < x_n \}$ (we put $\mathcal{X}_{1,1} = \mathbb{R}$) and $\mathcal{X}_n \coloneqq \bigcup \limits_{m=1}^n \mathcal{X}_{n,m}$. Let $\mathcal{H}_0 \coloneqq \mathbb{C} \, \Lambda$, $\mathcal{H}_n \coloneqq \text{L}^2(\mathcal{X}_n, \diff x_1 \ldots \diff x_n)$, where $\Lambda$ is a unit vector and $\diff x_1 \ldots \diff x_n$ is the $n$-dimensional Lebesgue measure. Let us define the \emph{continuous V-monotone Fock space} by 
$$
\mathcal{CV} = \bigoplus \limits_{n=0}^{\infty} \mathcal{H}_n 
$$
with the canonical scalar product, namely
$$
\langle h_1, h_2 \rangle \coloneqq \int \limits_{\mathcal{X}_n} h_1(x_1, \ldots, x_n) \overline{h_2(x_1, \ldots, x_n)} \,  \diff x_1 \ldots \diff x_n \text{,}
$$
where $h_1, h_2 \in \mathcal{H}_n$. Let $\phi$ be the vector state associated with $\Lambda$.
\end{definition}

Now, we introduce the creation and annihilation operators on the Fock space.

\begin{definition}
Let $\mathbbm{1}_n, \mathbbm{1}_{n,m} \colon \mathbb{R}^n \mapsto \{ 0, 1 \}$ be the characteristic funtions of $\mathcal{X}_{n}$ and $\mathcal{X}_{n,m}$, respectively, and let $f \in \mathcal{H}_1$. We define the creation operator associated with $f$ by the continuous linear extension of
\begin{equation*}
\begin{cases}
(a(f) \, \Lambda)(x) = f(x) \\
(a(f) \, g)(x, x_1, \ldots, x_n) = \left[ f(x) \mathbbm{1}_{n+1}(x, x_1, \ldots, x_n) \right] g(x_1, \ldots, x_n) \text{,}
\end{cases}
\end{equation*}
where $g \in \mathcal{H}_n$ for $n>0$. 

\noindent The annihilation operator associated with $f$ is given by by the continuous linear extension of $a^*(f) \, \Lambda = 0$ and
\begin{equation*}
\begin{cases}
(a^*(f) \, g) = \int \limits_{- \infty}^{\infty} \diff x \, \overline{f(x)} g(x) \Lambda & \text{for $n=0$} \\
(a^*(f) \, g)(\vec{x}) = \left[ \mathbbm{1}_{n,1}(\vec{x}) \int \limits_{- \infty}^{x_1} \diff x + \int \limits_{x_1}^{\infty} \diff x \right] \, \overline{f(x)} g(x, \vec{x}) & \text{for $n>0$}
\end{cases} \text{,}
\end{equation*}
where $g \in \mathcal{H}_{n+1}$ and $\vec{x} = (x_1, \ldots, x_n)$.
\end{definition}

\begin{remark}
Let $\mathbbm{1}_{<}, \mathbbm{1}_{>} \colon \mathbb{R}^2 \mapsto \{ 0, 1 \}$ be defined by
$$
\mathbbm{1}_{<} \coloneqq \mathbbm{1}_{ \{ (u,v) \in \mathbb{R}^2: u < v \} } \qquad \text{and} \qquad \mathbbm{1}_{>}(x, y) \coloneqq \mathbbm{1}_{<}(y, x) \text{.}
$$
We have 
\begin{equation}
\label{eq:recurrenceForCharacteristicFunction}
\mathbbm{1}_{n+1}(x, \vec{x}) = \mathbbm{1}_{<}(x, x_1) \mathbbm{1}_{n,1}(\vec{x}) + \mathbbm{1}_{>}(x, x_1) \mathbbm{1}_n(\vec{x}) \text{.}
\end{equation}
\end{remark}

\begin{proposition}
An operator $a^*(f)$ is the adjoint of $a(f)$.
\end{proposition}

\begin{proof}
Let $g \in \mathcal{H}_{n+1}$ and $h \in \mathcal{H}_n$ for some $n > 0$. Then, by~\eqref{eq:recurrenceForCharacteristicFunction} and Fubini's theorem, we get
\begin{multline*}
\langle g, a(f) \, h \rangle = \int \limits_{\mathbb{R}^{n+1}} \diff x \diff x_1 \ldots \diff x_n \, \mathbbm{1}_{n+1} (x, \vec{x}) g(x, \vec{x}) \overline{f(x) h(\vec{x})} \\
= \int \limits_{\mathbb{R}^n} \diff x_1 \ldots \diff x_n \, \left[ \mathbbm{1}_{n,1}(\vec{x}) \int \limits_{- \infty}^{x_1} \diff x + \mathbbm{1}_{n}(\vec{x}) \int \limits_{x_1}^{\infty} \diff x \right] \, \overline{f(x)} g(x, \vec{x}) \overline{h(\vec{x})} \\
= \int \limits_{\mathcal{X}_n} \diff x_1 \ldots \diff x_n \, \left[ \mathbbm{1}_{n,1}(\vec{x}) \int \limits_{- \infty}^{x_1} \diff x + \int \limits_{x_1}^{\infty} \diff x \right] \, \overline{f(x)} g(x, \vec{x}) \overline{h(\vec{x})} = \langle a^*(f) \, g, h \rangle \text{,}
\end{multline*}
where $\vec{x} = (x_1, \ldots, x_n)$. The proof for $g \in \mathcal{H}_1$ and $g = \Lambda$ is omitted.
\end{proof}

\begin{definition}
By the \emph{V-monotone Gaussian operator} associated with $f \in \mathcal{H}_1$ we mean the operator of the form
$$
\omega(f) = a(f) + a^*(f) \text{.}
$$
If $f = \mathbbm{1}_{[0,1]}$, this operator will be called \emph{standard} and denoted by $\omega$. By $a$ and $a^*$ we denote the associated creation and annihilation operators, respectively. 
\end{definition}

The moment sequence for $\omega$ with respect to $\phi$ happens to agree with the moment sequence of the standard V-monotone Gaussian distribution. In order to show this fact, we need a discrete version of the V-monotone Fock space and operators defined on it.

\begin{definition}
Let $\{ \Omega \} \cup \{ e_i: i \in I \}$ be an orthonormal basis of some Hilbert space. Let us define the \emph{discrete V-monotone Fock space} by
$$
\mathcal{DV} = \mathbb{C} \, \Omega \oplus \bigoplus \limits_{n=1}^{\infty} \bigoplus \limits_{(i_1, \ldots, i_n) \in I_n} \mathbb{C} \, e_{i_1} \otimes \ldots \otimes e_{i_n} \text{,}
$$
with the canonical scalar product and the vector state $\varphi$ associated with $\Omega$.

We define creation and annihilation operators by the continuous linear extension of
$$
\begin{cases}
a_i \, \Omega = e_i \text{,} \\
a_i \, e_{i_1} \otimes \ldots \otimes e_{i_n} = \mathbbm{1}_{n+1}(i, i_1, \ldots, i_n) e_i \otimes e_{i_1} \otimes \ldots \otimes e_{i_n} \text{,}
\end{cases}
$$
and
$$
\begin{cases}
a^*_i \, \Omega = 0 \text{,} \\
a^*_i \, e_j = \delta_{i,j} \Omega \text{,} \\
a^*_i \, e_j \otimes e_{i_1} \otimes \ldots \otimes e_{i_n} = \delta_{ij} e_{i_1} \otimes \ldots \otimes e_{i_n} \text{,}
\end{cases}
$$
where $\delta_{i,j}$ is the Kronecker delta. Of course $a^*_i$, is the adjoint of $a_i$. Let $\omega_i = a_i + a^*_i$. For any natural number $N$, let $a(N) = \tfrac{1}{\sqrt{N}} \sum \limits_{i=1}^N a_i$, $a^*(N) = (a(N))^*$ and $\omega(N) = a(N) + a^*(N)$.
\end{definition}

\begin{proposition}
The family $(\omega_i)_{i=1}^{\infty}$ is V-monotone independent and consists of identically distributed non-commutative random variables with mean $0$ and variance $1$.
\end{proposition}

\begin{proof}
Let $\mathcal{K}_i = \mathbbm{C}^2$, $\Omega_i = (1,0)$, $e_i = (0,1)$ and let $\varphi_i$ be the vector state associated with $\Omega_i$. Then $\mathcal{DV}$ is isomorphic (as a Hilbert space) to the V-monotone product of $(\mathcal{K}_i, \Omega_i)_{i \in I}$. One can check that $\omega_i = \widetilde{\lambda}_i(T_i)$, where 
$$
T_i = \left[
\begin{array}{ll}
0 & 1 \\
1 & 0
\end{array}
\right] \text{.}
$$
V-monotone independence follows then from Theorem~\ref{thm:realization}. Of course, each $T_i$ has the distribution $\tfrac{1}{2} (\delta_{-1}+\delta_{1})$.
\end{proof}

\begin{corollary}
\label{corollary:discreteFockCLT}
The limit distribution of $\omega(N)$ (for $N \to \infty$) is the V-monotone standard Gaussian distribution.
\end{corollary}

Now we show that operators $a$ and $a^*$ can be approximated in some sense by $a(N)$ and $a^*(N)$, respectively. We need some auxiliary objects in order to prove this fact.

\begin{definition}
Let $\mathcal{D}_n$ be the set of all $g \in \mathcal{H}_n$ such that
$$
g(\vec{x}) = 
\begin{cases}
g_{n,m}(\vec{x}) & \text{if $\vec{x} \in \mathcal{X}_{n,m} \cap [0,1]^n$ for $m=1,\ldots,n$} \\
0 & \text{otherwise,}
\end{cases}
$$
where $\vec{x} = (x_1, \ldots, x_n)$ and $g_{n,m}$ is a polynomial. We put $\mathcal{D}_0 = \mathcal{H}_0$. Define vectors
\begin{align*}
v_0^N(\Lambda) = & \Omega \\
v_{n}^N(g) = & N^{-n/2} \sum \limits_{(i_1,\ldots,i_n) \in I_n(N)}  g(i_1/N,\ldots,i_n/N) e_{i_1} \otimes \ldots \otimes e_{i_n} \text{,}
\end{align*}
where $I_n(N) = I_n \cap [N]^n$ and $g \in \mathcal{D}_n$.
\end{definition}

\begin{lemma}
\label{lem:operatorApproximationx1}
For any $n>0$ and any $g \in \mathcal{D}_n$, we have
\begin{align*}
a(N) \, v_n^N(g) & = v_{n+1}^N(a \, g) \\
a^*(N) \, v_{n}^N(g) & = v^N_{n-1}(a^* \, g) + O(1/N) \text{,}
\end{align*}
where the symbol $O(1/N)$ means that the norm of the difference is smaller than $C/N$ for some $C>0$, independent of $N$.
\end{lemma}

\begin{proof}
The first statement follows directly from definitions of $a(N)$ and $v_n^N$. We prove the second one. Let $n \geq 0$ and let $g \in \mathcal{D}_{n+1}$. Then
\begin{multline*}
a^*(N) \, v_{n+1}^N(g) = \\
{N^{-n/2} \sum \limits_{(i_1, \ldots, i_n) \in I_n(N)} \left[ \mathbbm{1}_{n,1}(i_1, \ldots, i_n) \dfrac{1}{N} \sum \limits_{i=1}^{i_1-1} g(i/N, i_1/N, \ldots, i_n/N) \right.} \\
\left. + \dfrac{1}{N} \sum \limits_{i=i_1+1}^N g(i/N, i_1/N, \ldots, i_n/N) \right] e_{i_1} \otimes \ldots \otimes e_{i_n} \text{}
\end{multline*}
and
$$
(a^* \, g)(x_1, \ldots, x_n) = \left[ \mathbbm{1}_{n,1}(x_1, \ldots, x_n) \int \limits_0^{x_1} \diff x + \int \limits_{x_1}^1 \diff x \right] g(x, x_1, \ldots, x_n) \text{.}
$$
Let $\alpha_{i_1, \ldots, i_n}$ be the coefficient of $e_{i_1} \otimes \ldots \otimes e_{i_n}$ in $N^{n/2} (a^*(N) \, v_{n+1}^N(g) - v_{n}(a^* \, g))$. If $(i_1, \ldots, i_n) \in I_{n,m}(N)$ for $m>1$, then it has the following form:
\begin{equation*}
\dfrac{1}{N} \sum \limits_{i=i_1+1}^N g_{n+1,m+1}(i/N, i_1/N, \ldots, i_n/N) - \int \limits_{i_1/N}^{1} \diff x \, g_{n+1,m+1}(x, i_1/N, \ldots, i_n/N) \text{,}
\end{equation*}
which we rewrite as
\begin{equation}
\label{eq:coefficientOfSimpleTensor}
\sum \limits_{i=i_1+1}^N \int \limits_{(i-1)/N}^{i/N} \diff x \, \left[ g_{n+1,m+1}(i/N, i_1/N, \ldots, i_n/N) - g_{n+1,m+1}(x, i_1/N, \ldots, i_n/N) \right] \text{.}
\end{equation}
By the Lagrange Mean Value Theorem, we get
\begin{multline}
\label{eq:LagrangesMVEstimation}
|g_{n+1,m+1}(i/N, i_1/N, \ldots, i_n/N) - g_{n+1,m+1}(x, i_1/N, \ldots, i_n/N)| \\
= (x-i/N) |\partial_x \, g_{n+1,m+1}(c, i_1/N, \ldots, i_n/N)| \\
\leq \dfrac{\Vert \partial_x \, g_{n+1,m+1} \Vert_{\infty}}{N} \text{}
\end{multline}
for $i \in [i_1+1,N]$, a real number $x \in [(i-1)/N, i/N]$, and some $c$ between $x$ and $i/N$ (by $\Vert . \Vert_{\infty}$ we mean the maximum of a polynomial on $[0,1]^{n+1}$). Using the triangle inequality and~\eqref{eq:LagrangesMVEstimation}, we estimate the absolute value of~\eqref{eq:coefficientOfSimpleTensor} by
$$
\sum \limits_{i=i_1+1}^N \int \limits_{(i-1)/N}^{i/N} \diff x \, \dfrac{\Vert \partial_x \, g_{n+1,m+1} \Vert_{\infty}}{N} \leq \dfrac{\Vert \partial_x \, g_{n+1,m+1} \Vert_{\infty}}{N} \text{.}
$$
If $(i_1, \ldots, i_n) \in I_{n,1}$, then $\alpha_{i_1, \ldots, i_n}$ can be rewritten as
\begin{equation*}
\begin{split}
& \sum \limits_{i=1}^{i_1} \int \limits_{(i-1)/N}^{i/N} \diff x \, \left[ g_{n+1,1}(i/N, i_1/N, \ldots, i_n/N) -  g_{n+1,1}(x, i_1/N, \ldots, i_n/N) \right] \\
& - \dfrac{1}{N} g_{n+1,1}(i_1/N, i_1/N, \ldots, i_n/N) \\
& + \sum \limits_{i=i_1+1}^N \int \limits_{(i-1)/N}^{i/N} \diff x \, \left[ g_{n+1,2}(i/N, i_1/N, \ldots, i_n/N) - g_{n+1,2}(x, i_1/N, \ldots, i_n/N) \right] \text{.}
\end{split}
\end{equation*}
and we estimate its absolute value similarly by
$$
\dfrac{ \Vert g_{n+1,1} \Vert_{\infty} + \max(\Vert \partial_x \, g_{n+1,1} \Vert_{\infty}, \Vert \partial_x \, g_{n+1,2} \Vert_{\infty}) }{N} \text{.}
$$
Hence
\begin{align*}
\Vert a^*(N) \, v_{n+1}^N(g) - v_{n}(a^* \, g) \Vert^2 & = N^{-n} \sum \limits_{(i_1, \ldots, i_n) \in I_n(N)} |\alpha_{i_1, \ldots, i_n}|^2 \\
& \leq \dfrac{1}{N^2} \left( \Vert g_{n+1,1} \Vert_{\infty} + \max \limits_{m=1,\ldots,n+1} \Vert \partial_x \, g_{n+1,m} \Vert_{\infty} \right)^2 \text{,}
\end{align*}
which completes the proof.
\end{proof}

Now we will investigate how the operator $\omega(N)$ approximates $\omega$. 

\begin{definition}
For each non-crossing pair partition $\pi$, we define recursively the following operators:
\begin{align*}
a_{ \pi } & =
\begin{cases}
\id & \text{if $\pi = \{ \emptyset \}$} \\
a^* a_{\pi'} a a_{\pi''} & \text{if $\pi = \pairPartitionRecurrence{\pi'}{\pi''}$} 
\end{cases} \text{,} \\
a_{ \pi }(N) & =
\begin{cases}
\id & \text{if $\pi = \{ \emptyset \}$} \\
a^*(N) a_{\pi'}(N) a(N) a_{\pi''}(N) & \text{if $\pi = \pairPartitionRecurrence{\pi'}{\pi''}$}
\end{cases} \text{.}
\end{align*}
\end{definition}

\begin{lemma}
\label{lem:operatorApproximation2}
For any $n>0$, a function $g \in \mathcal{D}_n$ and every $\pi \in \noncr^2$, we have
\begin{align*}
a_{\pi}(N) \, \Omega & = \phi(a_{\pi}) \Omega + O(1/N) \\
a_{\pi}(N) \, v_n^N(g) & = v_n^N(a_{\pi} \, g) + O(1/N) \text{,}
\end{align*}
where the symbol $O(1/N)$ means that the norm of the difference is smaller than $C/N$ for some $C>0$, independent of $N$.
\end{lemma}

\begin{proof}
First note that the norms of $a(N)$, $a^*(N)$ and $a_{\pi}(N)$ are at most $1$. It follows immediately from the facts that $a_i$ is a partial isometry and for any $i \neq j$ we have $a_i \, \mathcal{DV} \perp a_j \, \mathcal{DV}$. Note also that the image of any $g \in \mathcal{D}_n$ under $a$ and under $a^*$ belongs to $\mathcal{D}_{n+1}$ and $\mathcal{D}_{n-1}$, respectively, for any positive integer $n$.

We will justify the second equality by induction on the number of blocks of $\pi$. Let $\pi', \pi'' \in \noncr^2$ be such that $\pi = \pairPartitionRecurrence{\pi'}{\pi''}$. We obtain from Lemma~\ref{lem:operatorApproximationx1} and from the induction hypothesis that
\begin{align*}
a_{\pi}(N) \, v_{n}^{N}(g) & = a^*(N) a_{\pi'}(N) a(N) a_{\pi''}(N) \, v_{n}^{N}(g) \\
& = a^*(N) a_{\pi'}(N) a(N) \, ( v_{n}^{N}(a_{\pi''} \, g ) + w_1 ) \\
& = a^*(N) a_{\pi'}(N) \, ( v_{n}^{N}(a a_{\pi''} \, g ) + a(N) \, w_1 ) \\
& = a^*(N) \, ( v_{n}^{N}( a_{\pi'} a a_{\pi''} \, g ) + w_2 + a_{\pi'}(N) a(N) \, w_1 ) \text{,}
\end{align*}
where $w_1, w_2 = O(1/N)$. Since $\Vert a_{\pi'}(N) a(N) \Vert \leq 1$, we get 
$$
w'_1 \coloneqq w_2 + a_{\pi'}(N) a(N) \, w_1 = O(1/N)
$$
and thus
\begin{multline*}
a_{\pi}(N) \, v_{n}^{N}(g) = v_{n}^{N}( a^* a_{\pi'} a a_{\pi''} \, g ) + w'_2 + a^*(N) \, w'_1 = v_{n}^{N}( a_{\pi} \, g ) + O(1/N) \text{}
\end{multline*}
(where $w'_2 = O(1/N)$), because $\Vert a^*(N) \Vert \leq 1$ and the induction is complete. The first equality can be proven similarly.
\end{proof}

Now we are ready to state and prove the main theorem of this section.

\begin{theorem}
\label{theorem:operator_model}
The random variable $\omega$ has the standard V-monotone Gaussian distribution.
\end{theorem}

\begin{proof}
We first show that 
\begin{align}
\label{eq:sumOnlyOverPairPartitions}
\begin{split}
\phi(\omega^{k}) = & \sum \limits_{\pi \in \noncrk{2}(k)} \phi(a_{\pi}) \text{,} \\
\varphi((\omega(N))^{k}) = & \sum \limits_{\pi \in \noncrk{2}(k)} \varphi(a_{\pi}(N)) \text{}
\end{split}
\end{align}
for any $k \in \mathbb{N}$. We adopt the convention that the value of the sum over the empty set is equal to zero. Let $\varepsilon \in \{ 1, * \}$. We will prove only the first equality, because the second one can be proven similarly. 
Observe that the image of $\mathcal{H}_{l}$ under $a$ and $a^*$ is a subset of $\mathcal{H}_{l+1}$ and $\mathcal{H}_{l-1}$ (we put $\mathcal{H}_{-1} = \{ 0 \}$), respectively. Define
$$
a^{\varepsilon} = 
\begin{cases}
a & \text{if $\varepsilon = 1$} \\
a^* & \text{if $\varepsilon = *$}
\end{cases} \text{}
\qquad \text{and} \qquad
\sgn(\varepsilon) = 
\begin{cases}
-1 & \text{if $\varepsilon = 1$} \\
1 & \text{if $\varepsilon = *$}
\end{cases} \text{.}
$$
Assume that $\varepsilon$'s in the operator $b = a^{\varepsilon_1} \ldots a^{\varepsilon_k}$ do not fulfill the conditions
\begin{equation}
\label{eq:catalanPathsConditions}
\begin{cases}
\sgn(\varepsilon_1) + \ldots + \sgn(\varepsilon_l) \geq 0 & \text{for all $1 \leq l < k$,} \\
\sgn(\varepsilon_1) + \ldots + \sgn(\varepsilon_n) = 0 \text{} & \\
\end{cases}
\end{equation}
(cf. conditions (2.6) in~\cite{NicaSpeicher2006}). Then there exists $l \in [k]$ such that $\sgn(\varepsilon_l) + \ldots + \sgn(\varepsilon_k) > 0$ and we take the largest such $l$. If $l > 1$, then we have
$$
a^{\varepsilon_l} \ldots a^{\varepsilon_k} \, \Lambda = a^* a^{\varepsilon_{l+1}} \ldots a^{\varepsilon_k} \, \Lambda = a^* \, C \Lambda = 0 \text{}
$$
for some constant $C \in \mathbb{C}$. For $l=1$ we have $b \, \Lambda \perp \Lambda$. In both cases, $\phi(b) = 0$. 

Otherwise there exists $\pi \in \noncrk{2}$ such that $b = a_{\pi}$, which we shall prove by induction. We may only consider $k=2n+2$ for some non-negative integer $n$. Let $m \in [n]$ be the smallest number such that $\sgn(\varepsilon_1) + \ldots + \sgn(\varepsilon_{2m+2}) = 0$. We have 
$$
b = a^* a^{\varepsilon_2} \ldots a^{\varepsilon_{2m+1}} a a^{\varepsilon_{2m+3}} \ldots a^{\varepsilon_{2n+2}} = a^* a_{\pi'} a a_{\pi''} = a_{\pi} \text{,}
$$
where $\noncrk{2}(2n+2) \ni \pi = \pairPartitionRecurrence{\pi'}{\pi''}$ for some $\pi' \in \noncrk{2}(2m)$ and $\pi'' \in \noncrk{2}(2n-2m)$, since both the sequences $(\varepsilon_2, \ldots, \varepsilon_{2m+1})$ and $(\varepsilon_{2m+3}, \ldots, \varepsilon_{2n+2})$ fulfill~\eqref{eq:catalanPathsConditions}. Of course every operator $a_{\pi}$ for $\pi \in \noncrk{2}(2n)$ has the form $a^{\varepsilon_l} \ldots a^{\varepsilon_{2n}}$ for some sequence $(\varepsilon_1, \ldots, \varepsilon_{2n})$ satisfying~\eqref{eq:catalanPathsConditions}.

Lemma~\ref{lem:operatorApproximation2} implies that for any $\pi \in \noncrk{2}$ 
$$
\lim \limits_{N \to \infty} \varphi( a_{\pi}(N) ) = \phi( a_{\pi} ) \text{.}
$$
Combining it with~\eqref{eq:sumOnlyOverPairPartitions}, we get
$$
\lim \limits_{N \to \infty} \varphi( (\omega(N))^k ) = \phi(\omega^k) \text{}
$$
for any natural number $k$. The assertion follows then from Corollary~\ref{corollary:discreteFockCLT}.
\end{proof}

\section{Moment generating function}
We investigate the distribution of $\omega$. First, we describe how an operator $a_{\pi}$ acts on functions $g \in \mathcal{D}_n$. We introduce two families of polynomials $(P_\pi)_{\pi \in \noncrk{2}}$, $(Q_{\pi})_{\pi \in \noncrk{2}}$ defined on $[0,1]$ recursively by
\begin{align*}
& \begin{cases}
P_{ \pi } = \mathbbm{1}_{[0,1]} & \text{if $\pi = \{ \emptyset \}$} \\
P_{ \pi }(x) = \left[ \int \limits_{0}^{x} P_{\pi'}(t) \, \diff t + \int \limits_{x}^{1} Q_{\pi'}(t) \, \diff t \right] P_{\pi''}(x) & \text{if $\pi = \pairPartitionRecurrence{\pi'}{\pi''}$}
\end{cases} \text{,} \\
& \begin{cases}
Q_{ \pi } = \mathbbm{1}_{[0,1]} & \text{if $\pi = \{ \emptyset \}$} \\
Q_{ \pi }(x) = \left[ \int \limits_{x}^{1} Q_{\pi'}(t) \, \diff t \right] Q_{\pi''}(x) & \text{if $\pi = \pairPartitionRecurrence{\pi'}{\pi''}$.}
\end{cases}
\end{align*}

\begin{remark}
\label{remark:polynomialsQexpressedByUsingMurakisNumbers}
We have $Q_{\pi}(x) = q_{\pi} (1-x)^{|\pi|}$ for any $\pi \in \noncrk{2}$ and any $x \in [0,1]$, where $|\pi|$ denotes the number of blocks of $\pi$ and where the sequence $(q_{\pi})_{\pi \in \noncrk{2}}$ is defined by the following recurrence relation:
$$
\begin{cases}
q_{ \pi } = 1 & \text{if $\pi = \{ \emptyset \}$} \\
q_{ \pi }(x) = \dfrac{q_{\pi'} q_{\pi''}}{|\pi'|+1} & \text{if $\pi = \pairPartitionRecurrence{\pi'}{\pi''}$}
\end{cases} \text{.}
$$
\end{remark}

\begin{proof}
Let $\pi = \pairPartitionRecurrence{\pi'}{\pi''}$. By induction on the number of blocks, we obtain
\begin{align*}
Q_{\pi}(x) = \left[ \int \limits_{x}^{1} Q_{\pi'}(t) \, \diff t \right] Q_{\pi''}(x) & = q_{\pi'} \left[ \int \limits_{x}^{1} (1-t)^{|\pi'|} \, \diff t \right] q_{\pi''}(1-x)^{|\pi''|} \\
& = \dfrac{q_{\pi'} (1-x)^{|\pi'|+1}}{|\pi'|+1} q_{\pi''}(1-x)^{|\pi''|} \\
& = q_{\pi} (1-x)^{|\pi|} \text{.}
\end{align*}
\end{proof}

\begin{example}
Some examples of $P_{\pi}$ and $Q_{\pi}$ for the smallest non-crossing pair partitions are presented in Fig.~\ref{figure:pairPartitionPolynomialsOfTypePandQ}. 
\begin{figure}[H]
\centering
\begin{tikzpicture}
    \tikzstyle{leg} = [circle, draw=black, fill=black!100, text=black!100, thin, inner sep=0pt, minimum size=2.0]

    \pgfmathsetmacro {\dx}{0.25}
    \pgfmathsetmacro {\dy}{0.40}

    \pgfmathsetmacro {\x}{2*\dx}
    \pgfmathsetmacro {\y}{0*\dy}

    \draw[-] (\x+0*\dx,\y+0*\dy) -- (\x+0*\dx,\y+1*\dy);
    \node () at (\x+0*\dx,\y+0*\dy) [leg] {.};
    \draw[-] (\x+1*\dx,\y+0*\dy) -- (\x+1*\dx,\y+1*\dy);
    \node () at (\x+1*\dx,\y+0*\dy) [leg] {.};
    \draw[-] (\x+0*\dx,\y+1*\dy) -- (\x+1*\dx,\y+1*\dy);
    
    \node () at (\x+3.5*\dx,\y+0*\dy) [label={[label distance = -1.5mm]right:\small{$Q_{\pi} = 1-x$}}] {};
    \node () at (\x+3.5*\dx,\y+2*\dy) [label={[label distance = -1.5mm]right:\small{$P_{\pi} = 1$}}] {};

    \pgfmathsetmacro {\x}{19*\dx}
    \pgfmathsetmacro {\y}{0*\dy}

    \draw[-] (\x+0*\dx,\y+0*\dy) -- (\x+0*\dx,\y+1*\dy);
    \node () at (\x+0*\dx,\y+0*\dy) [leg] {.};
    \draw[-] (\x+1*\dx,\y+0*\dy) -- (\x+1*\dx,\y+1*\dy);
    \node () at (\x+1*\dx,\y+0*\dy) [leg] {.};
    \draw[-] (\x+0*\dx,\y+1*\dy) -- (\x+1*\dx,\y+1*\dy);
    
    \draw[-] (\x+2*\dx,\y+0*\dy) -- (\x+2*\dx,\y+1*\dy);
    \node () at (\x+2*\dx,\y+0*\dy) [leg] {.};
    \draw[-] (\x+3*\dx,\y+0*\dy) -- (\x+3*\dx,\y+1*\dy);
    \node () at (\x+3*\dx,\y+0*\dy) [leg] {.};
    \draw[-] (\x+2*\dx,\y+1*\dy) -- (\x+3*\dx,\y+1*\dy);
    
    \node () at (\x+4.5*\dx,\y+0*\dy) [label={[label distance = -1.5mm]right:\small{$Q_{\pi} = (1-x)^2$}}] {};
    \node () at (\x+4.5*\dx,\y+2*\dy) [label={[label distance = -1.5mm]right:\small{$P_{\pi} = 1$}}] {};

    \pgfmathsetmacro {\x}{37*\dx}
    \pgfmathsetmacro {\y}{0*\dy}

    \draw[-] (\x+0*\dx,\y+0*\dy) -- (\x+0*\dx,\y+2*\dy);
    \node () at (\x+0*\dx,\y+0*\dy) [leg] {.};
    \draw[-] (\x+3*\dx,\y+0*\dy) -- (\x+3*\dx,\y+2*\dy);
    \node () at (\x+3*\dx,\y+0*\dy) [leg] {.};
    \draw[-] (\x+0*\dx,\y+2*\dy) -- (\x+3*\dx,\y+2*\dy);
    
    \draw[-] (\x+1*\dx,\y+0*\dy) -- (\x+1*\dx,\y+1*\dy);
    \node () at (\x+1*\dx,\y+0*\dy) [leg] {.};
    \draw[-] (\x+2*\dx,\y+0*\dy) -- (\x+2*\dx,\y+1*\dy);
    \node () at (\x+2*\dx,\y+0*\dy) [leg] {.};
    \draw[-] (\x+1*\dx,\y+1*\dy) -- (\x+2*\dx,\y+1*\dy);
    
    \node () at (\x+4.5*\dx,\y+0*\dy) [label={[label distance = -1.5mm]right:\small{$Q_{\pi} = \tfrac{(1-x)^2}{2}$}}] {};
    \node () at (\x+4.5*\dx,\y+2*\dy) [label={[label distance = -1.5mm]right:\small{$P_{\pi} = \tfrac{x^2+1}{2}$}}] {};

    \pgfmathsetmacro {\x}{0*\dx}
    \pgfmathsetmacro {\y}{-6*\dy}

    \draw[-] (\x+0*\dx,\y+0*\dy) -- (\x+0*\dx,\y+1*\dy);
    \node () at (\x+0*\dx,\y+0*\dy) [leg] {.};
    \draw[-] (\x+1*\dx,\y+0*\dy) -- (\x+1*\dx,\y+1*\dy);
    \node () at (\x+1*\dx,\y+0*\dy) [leg] {.};
    \draw[-] (\x+0*\dx,\y+1*\dy) -- (\x+1*\dx,\y+1*\dy);
    
    \draw[-] (\x+2*\dx,\y+0*\dy) -- (\x+2*\dx,\y+1*\dy);
    \node () at (\x+2*\dx,\y+0*\dy) [leg] {.};
    \draw[-] (\x+3*\dx,\y+0*\dy) -- (\x+3*\dx,\y+1*\dy);
    \node () at (\x+3*\dx,\y+0*\dy) [leg] {.};
    \draw[-] (\x+2*\dx,\y+1*\dy) -- (\x+3*\dx,\y+1*\dy);
    
    \draw[-] (\x+4*\dx,\y+0*\dy) -- (\x+4*\dx,\y+1*\dy);
    \node () at (\x+4*\dx,\y+0*\dy) [leg] {.};
    \draw[-] (\x+5*\dx,\y+0*\dy) -- (\x+5*\dx,\y+1*\dy);
    \node () at (\x+5*\dx,\y+0*\dy) [leg] {.};
    \draw[-] (\x+4*\dx,\y+1*\dy) -- (\x+5*\dx,\y+1*\dy);
    
    \node () at (\x+5.5*\dx,\y+0*\dy) [label={[label distance = -1.5mm]right:\small{$Q_{\pi} = (1-x)^3$}}] {};
    \node () at (\x+5.5*\dx,\y+2*\dy) [label={[label distance = -1.5mm]right:\small{$P_{\pi} = 1$}}] {};

    \pgfmathsetmacro {\x}{18*\dx}
    \pgfmathsetmacro {\y}{-6*\dy}

    \draw[-] (\x+0*\dx,\y+0*\dy) -- (\x+0*\dx,\y+1*\dy);
    \node () at (\x+0*\dx,\y+0*\dy) [leg] {.};
    \draw[-] (\x+1*\dx,\y+0*\dy) -- (\x+1*\dx,\y+1*\dy);
    \node () at (\x+1*\dx,\y+0*\dy) [leg] {.};
    \draw[-] (\x+0*\dx,\y+1*\dy) -- (\x+1*\dx,\y+1*\dy);
    
    \draw[-] (\x+2*\dx,\y+0*\dy) -- (\x+2*\dx,\y+2*\dy);
    \node () at (\x+2*\dx,\y+0*\dy) [leg] {.};
    \draw[-] (\x+5*\dx,\y+0*\dy) -- (\x+5*\dx,\y+2*\dy);
    \node () at (\x+5*\dx,\y+0*\dy) [leg] {.};
    \draw[-] (\x+2*\dx,\y+2*\dy) -- (\x+5*\dx,\y+2*\dy);
    
    \draw[-] (\x+3*\dx,\y+0*\dy) -- (\x+3*\dx,\y+1*\dy);
    \node () at (\x+3*\dx,\y+0*\dy) [leg] {.};
    \draw[-] (\x+4*\dx,\y+0*\dy) -- (\x+4*\dx,\y+1*\dy);
    \node () at (\x+4*\dx,\y+0*\dy) [leg] {.};
    \draw[-] (\x+3*\dx,\y+1*\dy) -- (\x+4*\dx,\y+1*\dy);
    
    \node () at (\x+5.5*\dx,\y+0*\dy) [label={[label distance = -1.5mm]right:\small{$Q_{\pi} = (1-x)^3$}}] {};
    \node () at (\x+5.5*\dx,\y+2*\dy) [label={[label distance = -1.5mm]right:\small{$P_{\pi} = 1$}}] {};

    \pgfmathsetmacro {\x}{36*\dx}
    \pgfmathsetmacro {\y}{-6*\dy}

    \draw[-] (\x+0*\dx,\y+0*\dy) -- (\x+0*\dx,\y+2*\dy);
    \node () at (\x+0*\dx,\y+0*\dy) [leg] {.};
    \draw[-] (\x+3*\dx,\y+0*\dy) -- (\x+3*\dx,\y+2*\dy);
    \node () at (\x+3*\dx,\y+0*\dy) [leg] {.};
    \draw[-] (\x+0*\dx,\y+2*\dy) -- (\x+3*\dx,\y+2*\dy);
    
    \draw[-] (\x+1*\dx,\y+0*\dy) -- (\x+1*\dx,\y+1*\dy);
    \node () at (\x+1*\dx,\y+0*\dy) [leg] {.};
    \draw[-] (\x+2*\dx,\y+0*\dy) -- (\x+2*\dx,\y+1*\dy);
    \node () at (\x+2*\dx,\y+0*\dy) [leg] {.};
    \draw[-] (\x+1*\dx,\y+1*\dy) -- (\x+2*\dx,\y+1*\dy);
    
    \draw[-] (\x+4*\dx,\y+0*\dy) -- (\x+4*\dx,\y+1*\dy);
    \node () at (\x+4*\dx,\y+0*\dy) [leg] {.};
    \draw[-] (\x+5*\dx,\y+0*\dy) -- (\x+5*\dx,\y+1*\dy);
    \node () at (\x+5*\dx,\y+0*\dy) [leg] {.};
    \draw[-] (\x+4*\dx,\y+1*\dy) -- (\x+5*\dx,\y+1*\dy);
    
    \node () at (\x+5.5*\dx,\y+0*\dy) [label={[label distance = -1.5mm]right:\small{$Q_{\pi} = (1-x)^3$}}] {};
    \node () at (\x+5.5*\dx,\y+2*\dy) [label={[label distance = -1.5mm]right:\small{$P_{\pi} = 1$}}] {};
    
    \pgfmathsetmacro {\x}{0}
    \pgfmathsetmacro {\y}{-13*\dy}

    \draw[-] (\x+0*\dx,\y+0*\dy) -- (\x+0*\dx,\y+3*\dy);
    \node () at (\x+0*\dx,\y+0*\dy) [leg] {.};
    \draw[-] (\x+5*\dx,\y+0*\dy) -- (\x+5*\dx,\y+3*\dy);
    \node () at (\x+5*\dx,\y+0*\dy) [leg] {.};
    \draw[-] (\x+0*\dx,\y+3*\dy) -- (\x+5*\dx,\y+3*\dy);
    
    \draw[-] (\x+1*\dx,\y+0*\dy) -- (\x+1*\dx,\y+1*\dy);
    \node () at (\x+1*\dx,\y+0*\dy) [leg] {.};
    \draw[-] (\x+2*\dx,\y+0*\dy) -- (\x+2*\dx,\y+1*\dy);
    \node () at (\x+2*\dx,\y+0*\dy) [leg] {.};
    \draw[-] (\x+1*\dx,\y+1*\dy) -- (\x+2*\dx,\y+1*\dy);
    
    \draw[-] (\x+3*\dx,\y+0*\dy) -- (\x+3*\dx,\y+1*\dy);
    \node () at (\x+3*\dx,\y+0*\dy) [leg] {.};
    \draw[-] (\x+4*\dx,\y+0*\dy) -- (\x+4*\dx,\y+1*\dy);
    \node () at (\x+4*\dx,\y+0*\dy) [leg] {.};
    \draw[-] (\x+3*\dx,\y+1*\dy) -- (\x+4*\dx,\y+1*\dy);
    
    \node () at (\x+5.5*\dx,\y+0*\dy) [label={[label distance = -1.5mm]right:\small{$Q_{\pi} = \tfrac{(1-x)^3}{3}$}}] {};
    \node () at (\x+5.5*\dx,\y+2*\dy) [label={[label distance = -1.5mm]right:\small{$P_{\pi} = \tfrac{-x^3 + 3x^2 + 1}{3}$}}] {};
    
    \pgfmathsetmacro {\x}{18*\dx}
    \pgfmathsetmacro {\y}{-13*\dy}

    \draw[-] (\x+0*\dx,\y+0*\dy) -- (\x+0*\dx,\y+3*\dy);
    \node () at (\x+0*\dx,\y+0*\dy) [leg] {.};
    \draw[-] (\x+5*\dx,\y+0*\dy) -- (\x+5*\dx,\y+3*\dy);
    \node () at (\x+5*\dx,\y+0*\dy) [leg] {.};
    \draw[-] (\x+0*\dx,\y+3*\dy) -- (\x+5*\dx,\y+3*\dy);
    
    \draw[-] (\x+1*\dx,\y+0*\dy) -- (\x+1*\dx,\y+2*\dy);
    \node () at (\x+1*\dx,\y+0*\dy) [leg] {.};
    \draw[-] (\x+4*\dx,\y+0*\dy) -- (\x+4*\dx,\y+2*\dy);
    \node () at (\x+4*\dx,\y+0*\dy) [leg] {.};
    \draw[-] (\x+1*\dx,\y+2*\dy) -- (\x+4*\dx,\y+2*\dy);
    
    \draw[-] (\x+2*\dx,\y+0*\dy) -- (\x+2*\dx,\y+1*\dy);
    \node () at (\x+2*\dx,\y+0*\dy) [leg] {.};
    \draw[-] (\x+3*\dx,\y+0*\dy) -- (\x+3*\dx,\y+1*\dy);
    \node () at (\x+3*\dx,\y+0*\dy) [leg] {.};
    \draw[-] (\x+2*\dx,\y+1*\dy) -- (\x+3*\dx,\y+1*\dy);
    
    \node () at (\x+5.5*\dx,\y+0*\dy) [label={[label distance = -1.5mm]right:\small{$Q_{\pi} = \tfrac{(1-x)^3}{6}$}}] {};
    \node () at (\x+5.5*\dx,\y+2*\dy) [label={[label distance = -1.5mm]right:\small{$P_{\pi} = \tfrac{3x^2 + 1}{6}$}}] {};
    
\end{tikzpicture}
\caption{$P_{\pi}$ and $Q_{\pi}$ for the smallest non-crossing pair partitions}
\label{figure:pairPartitionPolynomialsOfTypePandQ}
\end{figure}
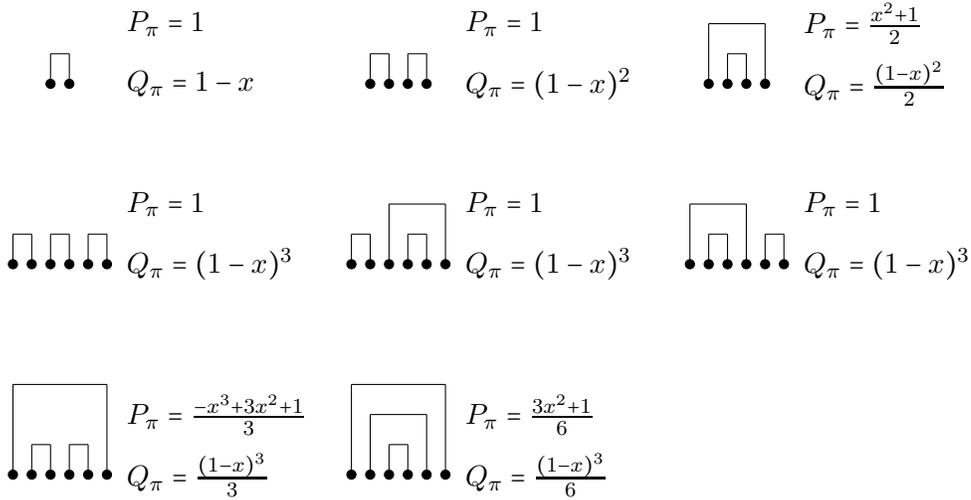

\noindent For instance,
$$
P_{\pi}(x) = \left[ \int \limits_{0}^{x} \tfrac{t^2+1}{2} \, \diff t + \int \limits_{x}^{1} \tfrac{(1-t)^2}{2} \, \diff t \right] \cdot 1 = \dfrac{3x^2 + 1}{6} \text{,}
$$
where $\pi = \{ \{ 1,6 \}, \{ 2,5 \}, \{ 3,4 \} \}$.
\end{example}

\begin{lemma}
\label{lemma:actingOfPairPartitionOperator}
Let $g \in \mathcal{D}_n$ for $n \in \mathbb{N}_+$ and $\mathbbm{1}'_n = \mathbbm{1}_n - \mathbbm{1}_{n,1}$. Then, for any $\pi \in \noncrk{2}$, we have
\begin{align*}
a_{ \pi } \, \Lambda & = P_{ \pi }(1) \Lambda \\
(a_{ \pi } \, g)(\vec{x}) & = [ P_{ \pi }(x_1) \mathbbm{1}_{n,1}(\vec{x}) + Q_{ \pi }(x_1)\mathbbm{1}'_n(\vec{x}) ] g(\vec{x}) \text{,}
\end{align*}
where $\vec{x} = (x_1, \ldots, x_n)$.
\end{lemma}

\begin{proof}
We will prove only the second equality, by induction on the number of blocks of a pair partition, the proof of the first one is similar. Let $\pi = \pairPartitionRecurrence{\pi'}{\pi''}$. We have
\begin{align*}
(a_{ \pi } g)(\vec{x}) = & \left[ \mathbbm{1}_{n,1}(\vec{x}) \int \limits_{0}^{x_1} \diff x + \int \limits_{x_1}^{1} \diff x \right] \, \mathbbm{1}_{n+1}(x,\vec{x}) (a_{ \pi' } a a_{ \pi'' }g)(x,\vec{x}) \\
= & \int \limits_0^1 \diff x \, \mathbbm{1}_{n+1}(x,\vec{x}) \left[ P_{ \pi' }(x) \mathbbm{1}_{n+1,1}(x,\vec{x}) + Q_{ \pi' }(x)\mathbbm{1}'_{n+1}(x,\vec{x}) \right] \mathbbm{1}_{n+1}(x,\vec{x}) \\ 
& \cdot \left[ P_{ \pi'' }(x_1) \mathbbm{1}_{n,1}(\vec{x}) + Q_{ \pi'' }(x_1)\mathbbm{1}'_n(\vec{x}) \right] g(\vec{x}) \\
= & \int \limits_0^1 \diff x \, \left[ P_{ \pi' }(x) \mathbbm{1}_{<}(x,x_1) \mathbbm{1}_{n,1}(\vec{x}) + Q_{ \pi' }(x) \mathbbm{1}_{>}(x,x_1) \mathbbm{1}_{n}(\vec{x}) \right] \\ 
& \cdot \left[ P_{ \pi'' }(x_1) \mathbbm{1}_{n,1}(\vec{x}) + Q_{ \pi'' }(x_1)\mathbbm{1}'_n(\vec{x}) \right] g(\vec{x}) \\
= & \int \limits_0^{x_1} \diff x \, \left[ P_{ \pi' }(x) P_{ \pi'' }(x_1) \mathbbm{1}_{n,1}(\vec{x}) \right] g(\vec{x}) + \int \limits_{x_1}^1 \diff x \, \left[ Q_{ \pi' }(x) P_{ \pi'' }(x_1) \mathbbm{1}_{n,1}(\vec{x}) \right] g(\vec{x}) \\
& + \int \limits_{x_1}^1 \diff x \, \left[ Q_{ \pi' }(x) Q_{ \pi'' }(x_1) \mathbbm{1}'_{n}(\vec{x}) \right] g(\vec{x}) \\
= & \ \mathbbm{1}_{n,1}(\vec{x}) \left[ \int \limits_0^{x_1} P_{ \pi' }(x) \, \diff x + \int \limits_{x_1}^1 Q_{ \pi' }(x) \, \diff x \right] P_{ \pi'' }(x_1) g(\vec{x}) \\
& + \mathbbm{1}'_n(\vec{x}) \left[ \int \limits_{x_1}^1 Q_{ \pi' }(x) \, \diff x \right] Q_{ \pi'' }(x_1) g(\vec{x}) \\
= & \ \left[ \mathbbm{1}_{n,1}(\vec{x}) P_{ \pi }(x_1) + \mathbbm{1}'_n(\vec{x}) Q_{ \pi }(x_1) \right] g(\vec{x}),
\end{align*}
which is our assertion.
\end{proof}

\begin{definition}
\label{definition:momentPolynomials}
Let us introduce two families of real-valued polynomials $(P_n)_{n=0}^{\infty}$ and $(Q_n)_{n=0}^{\infty}$, defined on $[0,1]$ recursively by 
\begin{align*}
& \begin{cases}
Q_0 = \mathbbm{1}_{[0,1]} \\
Q_{n+1}(x) = \sum \limits_{m=0}^n \left[ \int \limits_{x}^{1} Q_m(t) \, \diff t \right] Q_{n-m}(x) \text{,}
\end{cases} \\
& \begin{cases}
P_0 = \mathbbm{1}_{[0,1]} \\
P_{n+1}(x) = \sum \limits_{m=0}^n \left[ \int \limits_{0}^{x} P_m(t) \, \diff t + \int \limits_{x}^{1} Q_m(t) \, \diff t \right] P_{n-m}(x) \text{.}
\end{cases}
\end{align*}
\end{definition}

\begin{theorem}
\label{thm:distributionOfGaussianOperator}
For any natural number $n$, we have
$$
\phi(\omega^{2n}) = P_n(1) \qquad \text{and} \qquad \phi(\omega^{2n+1}) = 0 \text{.}
$$
\end{theorem}

\begin{proof}
By induction we get
\begin{equation}
\label{eq:polynomialsPandQasSumsIndexedByNoncrossingPairPartitions}
Q_n(x) = \sum \limits_{\pi \in \noncrk{2}(2n)} Q_{\pi}(x) \qquad
\text{and} \qquad
P_n(x) = \sum \limits_{\pi \in \noncrk{2}(2n)} P_{\pi}(x) \text{.}
\end{equation}
The assertion follows then from~\eqref{eq:sumOnlyOverPairPartitions} and Lemma~\ref{lemma:actingOfPairPartitionOperator}.
\end{proof}

Now we want to investigate the moment generating function for the operator $\omega$, namely $M(z) = \sum \limits_{n=0}^\infty \varphi(\omega^n) z^n$. Let
$$
f(z, x) = \sum \limits_{n=0}^\infty P_n(x) z^n \text{.}
$$
Since $M(z) = f(z^2, 1)$, we state and prove several facts about $f(z,x)$.

\begin{proposition}
For any $0 \leq |z| < \frac{1}{4}$ and any $x \in [0,1]$, the following integral equation holds:
\begin{equation}
\label{eq:integralEquation}
f(z,x) = 1 + z \left[ \int \limits_{0}^{x} f(z, t) \, \diff t + \dfrac{1-\sqrt{2zx+1-2z}}{z} \right] f(z,x) \text{.}
\end{equation}
\end{proposition}

\begin{proof}
Let $g(z, x) = \sum \limits_{n=0}^{\infty} Q_n(x) z^n$. We first show that
$$
\int \limits_{x}^1 g(z,t) \, \diff t = \dfrac{1-\sqrt{2zx+1-2z}}{z}
$$
for $|z(1-x)| < \tfrac{1}{2}$ and $z \neq 0$.
From~\eqref{eq:polynomialsPandQasSumsIndexedByNoncrossingPairPartitions} and Remark~\ref{remark:polynomialsQexpressedByUsingMurakisNumbers} that $Q_n(x) = q_n (1-x)^n$, where $(q_n)_{n=1}^{\infty}$ is the sequence of even moments of the standard arcsine distribution (i.e. $q_n = 2^{-n} { {2n}\choose{n} }$) satisfying the following recurrence relation:
$$
\begin{cases}
q_0 = 1 \\
q_{n+1} = \sum \limits_{m=0}^n \tfrac{1}{m+1} q_m q_{n-m} \text{.}
\end{cases}
$$
Therefore,
$$
g(z,x) = \dfrac{1}{\sqrt{2zx + 1 - 2z}} \text{.}
$$
Now, we will show by induction that 
$$
|P_n(x)| \leq C_n,
$$
where $(C_n)_{n=0}^\infty$ are Catalan numbers, which satisfy the following recurrence relation:
$$
\begin{cases}
C_0 = 1 \\
C_{n+1} = \sum \limits_{k=0}^n C_k C_{n-k}
\end{cases} \text{.}
$$
Using the fact that $q_n \leq C_n$ and the induction hypothesis, we have
\begin{align*}
|P_{n+1}(x)| \leq & \sum \limits_{k=0}^n \left[ \int \limits_{0}^{x} |P_k(t)| \, \diff t + \int \limits_{x}^{1} |Q_k(t)| \, \diff t \right] |P_{n-k}(x)| \\
\leq & \sum \limits_{k=0}^n \left[ x C_k + (1-x)C_k \right] C_{n-k} = C_{n+1} \text{.}
\end{align*}
The assertion follows then from the above inequality and Definition~\ref{definition:momentPolynomials}.
\end{proof}

Now we will give the solution of~\eqref{eq:integralEquation} in the implicit form. 

\begin{definition}
For $0 < t \leq \exp \big(\tfrac{\sqrt{3} \pi}{9} \big)$, let $S(t)$ be the inverse function of $\exp(T(t))$ on $[0, \infty)$, where
$$
T(t) = \dint \limits_{t}^{1} \dfrac{s \, \diff s}{s^2 - s + 1} = - \tfrac{1}{2} \log(t^2 - t + 1) - \tfrac{ \sqrt{3} }{3} \left[ \arctan \left( \tfrac{2t-1}{ \sqrt{3} } \right) - \tfrac{\pi}{6} \right] \text{.}
$$
The function $S$ is well defined, since for $t < 0$ we have $T'(t) = \tfrac{-t}{t^2-t+1} < 0$.
\end{definition}

\begin{lemma}
\label{lemma:solvingODE}
The solution $f$ of the integral equation~\eqref{eq:integralEquation} is given by
$$
f(z,x) = \left[ \sqrt{2zx + 1-2z} S \left( \sqrt{ \tfrac{2zx + 1-2z}{1-2z} } \right) \right]^{-1}
$$
for $z \in \left( 0, \tfrac{1}{4} \right)$ and $x \in [0, 1]$.
\end{lemma}

\begin{proof}
Fix $z \in \left( 0, \tfrac{1}{4} \right)$. Let $y(x)=\int \limits_0^x f(z,t) \, \diff t$. The equation~\eqref{eq:integralEquation} is equivalent to the initial value problem of the form
\begin{equation}
\label{eq:IVP1}
\begin{cases}
y' \left( \sqrt{2zx + 1-2z} - z y \right) = 1 \\
y(0) = 0
\end{cases}, \quad y=y(x), \quad x \in [0,1],
\end{equation}
which is an Abel ordinary differential equation of the second kind with an initial condition. 

First, we change coordinates
$$
\begin{cases}
x = \dfrac{\xi^2 + 2z - 1}{2z} \\
y = \dfrac{\xi - u(\xi)}{z}
\end{cases}
$$
for $\xi \geq 0$. The inverse mapping is given by
$$
\begin{cases}
\xi = \sqrt{2zx + 1 - 2z} \\
u = \xi - z y(x) \text{.}
\end{cases}
$$
After differentiation with respect to $\xi$, our ODE takes the form 
$$
\dfrac{1-u'}{\xi} \left( \xi - z \dfrac{\xi - u}{z} \right) = 1 \text{,}
$$
where $u' = u'(\xi) = \tfrac{\diff}{\diff\xi}u(\xi)$. Our new variable $\xi$ belongs to the interval $[\sqrt{1-2z}, 1]$ and the initial condition has takes form
$$
u(\sqrt{1-2z}) = \sqrt{1-2z} \text{.}
$$
Therefore, our initial value problem takes on the form
\begin{equation}
\label{eq:IVP2}
\begin{cases}
u u' - u = -\xi \\
u(\sqrt{1-2z}) = \sqrt{1-2z} \text{.}
\end{cases}
\end{equation}

The general solution of the above equation can be found in~\cite{PolyaninZaitsev2002} (under the number 1.3.1.2.) and has the parametric form
$$
\begin{cases}
\xi(t) = Ce^{T(t)} \\
u(t) = t \xi(t) \text{.}
\end{cases}
$$
Since the curve given above must contain the point $(\sqrt{1-2z}, \sqrt{1-2z})$, we conclude that $C=\sqrt{1-2z}$ and $t \geq 0$ and thus the solution of problem~\eqref{eq:IVP2} is given by
\begin{equation}
\label{eq:solutionOfModifiedODE}
u(\xi) = \xi S \left( \tfrac{\xi}{\sqrt{1-2z}} \right) \text{.}
\end{equation}
Since for every $z \in \left( 0, \tfrac{1}{4} \right)$ we have $\sqrt{1-2z} e^{T(0)} > 1$, we conclude that $[\sqrt{1-2z}, 1]$ is a subset of $(0, \sqrt{1-2z} e^{T(0)}]$, the domain of $S$. Combining~\eqref{eq:IVP2} with~\eqref{eq:solutionOfModifiedODE}, we obtain 
$$
u'(\xi) = 1-\dfrac{1}{S \left( \tfrac{\xi}{\sqrt{1-2z}} \right)} \text{,}
$$
which, after going back to the coordinates $(x, y)$, gives our assertion.
\end{proof}

\begin{corollary}
The moment generating function of the standard V-monotone Gaussian distribution has the implicit form
\begin{equation}
\label{eq:CLTMomentGeneratingFunction}
M(z) = 
\begin{cases}
1 & \text{if $z = 0$} \\
\left[ S \left( \dfrac{1}{\sqrt{ 1-2z^2 }} \right) \right]^{-1} & \text{otherwise}
\end{cases}
\end{equation}
for $z \in \left( -\tfrac{1}{2}, \tfrac{1}{2} \right)$.
\end{corollary}

From the technical point of view, our central limit distribution is not easy to handle. So far, we have only found the moment generating function in the implicit form. The situation resembles that of the distribution of $TT^*$, where $T$ is the triangular operator, studied by Dykema and Haagerup~\cite{DykemaHaagerup2005}, who have found the density function in the implicit form.

\section*{Acknowledgements}
I would like to thank Professor Romuald Lenczewski for suggesting the problem and his continuous help during the preparation of this paper. I would also like to thank the referees for carefully reading the paper as well as for their suggestions and many valuable comments. I am also grateful to my colleagues, Dariusz Kosz and Paweł Plewa for several helpful remarks.

This work has been supported by the Wroc\l{}aw University of Science and Technology (Grant number 0401/0121/17).


\end{document}